\documentclass[12pt, reqno]{amsart}
\usepackage{lipsum}
\usepackage{a4wide}
\usepackage{tikz} 
\usepackage{amscd}
\usepackage{hyperref}
\newcommand{\sslash}{\mathbin{/\mkern-6mu/}}
\usepackage{amsmath, amssymb, amsthm}
\usepackage{stmaryrd}
\usepackage{enumitem}
\hypersetup{colorlinks,linkcolor={blue},citecolor={blue},urlcolor={red}}
\theoremstyle{plain}
\usepackage{comment}
\numberwithin{equation}{section}

\newcommand{\diag}{\operatorname{diag}}

\newtheorem{theorem}{Theorem}[section]
\newtheorem{corollary}[theorem]{Corollary} 
\newtheorem{lemma}[theorem]{Lemma}
\theoremstyle{definition}
\newtheorem{remark}{Remark}
\newtheorem{proposition}[theorem]{Proposition}

\newtheorem{definition}[theorem]{Definition}

\newtheorem{example}{Example}

\usepackage{parskip}
\newcommand\blfootnote[1]{%
  \begingroup
  \renewcommand\thefootnote{}\footnote{#1}%
  \addtocounter{footnote}{-1}%
  \endgroup
}

\usepackage[utf8]{inputenc}

\title{Geometric Invariant Theory of Peterson Varieties}

\author[A.~Ghosh]{Arkadev Ghosh}
\address{Arkadev Ghosh \\
Chennai Mathematical Institute, Plot H1, SIPCOT IT Park, Siruseri, Kelambakkam,
603103, India.
}
\email{arkadev@cmi.ac.in}

\author[S.~Pattanayak]{Santosha Pattanayak}
\address{Santosha Pattanayak\\
Department of Mathematics, Indian Institute of Technology Kanpur, U.P. India, 208016.
}
\email{santosha@iitk.ac.in}

\begin{document}

\maketitle

\begin{abstract}
We study the GIT quotients of the Peterson variety
\(\mathrm{Pet}_n\subset \mathrm{GL}(n,\mathbb C)/B\) under a one-parameter subgroup $\lambda:\mathbb G_m \to T$ with respect to the linearization $\mathcal L(\chi)$ given by a regular dominant character $\chi$ in the root lattice. Using the Richardson stratification, we describe the semistable and 
stable loci explicitly in terms of subsets of simple
roots. This determines the GIT chamber decomposition and the
corresponding wall-crossing morphisms. In the deep chamber, the
quotient is shown to be isomorphic to the weighted projective space
\(\mathbb P(1,2,\ldots,n-1)\). We obtain a complete chamber-theoretic characterization of normality and describe how normality varies with the choice of linearization. We also
prove that the quotient is smooth if and only if \(n\le3\),
independently of the regular dominant linearization. These results
describe how the singular geometry of the Peterson variety is
reflected in the variation of its GIT quotients.
\end{abstract}


\section{Introduction}

The Peterson variety $\mathrm{Pet}_n$ occupies a central position at the intersection of geometry, representation theory, and combinatorics. In Lie type $A$, it is defined as the closed subvariety of the full flag variety $\mathrm{Fl}_n$ consisting of flags 
\[
0 = V_0 \subset V_1 \subset \cdots \subset V_n = \mathbb{C}^n
\]
such that $N V_i \subset V_{i+1}$ for each $1 \leq i \leq n-1$, where $N$ is a fixed regular nilpotent matrix (typically represented as a single principal Jordan block of size $n$).

The foundational importance of $\mathrm{Pet}_n$ stems from Peterson's remarkable unpublished observation that the quantum cohomology ring of the flag variety $\mathrm{GL}(n,\mathbb{C})/B$ is isomorphic to the coordinate ring of the intersection of the Peterson variety with the open opposite Schubert cell. This result was subsequently proved by Kostant \cite{Kos} for the full flag variety and extended by Rietsch \cite{Riet} to partial flag varieties, thereby establishing the Peterson variety as a fundamental bridge between classical Schubert calculus and quantum cohomology. As a regular nilpotent Hessenberg variety, $\mathrm{Pet}_n$ also occupies a central position within the theory of Hessenberg varieties, introduced by De Mari, Procesi, and Shayman (\cite{MPS}), and has deep connections with geometric representation theory, the Toda lattice, the affine Grassmannian, Springer theory, and the combinatorics of the Bruhat order (see, for example, \cite{Bal, Kos, Riet}). Peterson varieties may also be realized as a flat degeneration of certain regular semisimple Hessenberg varieties such as the permutohedral variety (see \cite{AFZ}).

Despite its central role in geometry and representation theory, the Peterson variety exhibits remarkably rich and subtle singular behavior. Although $\mathrm{Pet}_n$ is irreducible and of comparatively simple dimension, it is singular for every $n\ge 3$, and is known to be non-normal for $n\ge 4$. A major advance in understanding its local geometry was achieved by Insko and Yong through their introduction of patch ideals, which provide explicit local equations for neighborhoods of points in $\mathrm{Pet}_n$. Using this approach, they proved that $\mathrm{Pet}_n$ is a local complete intersection and therefore Cohen-Macaulay and Gorenstein, and they obtained a complete combinatorial characterization of its singular locus (see \cite{Insko}).

Alongside these developments, the topology, Schubert calculus, and equivariant geometry of Peterson varieties have been studied extensively. Tymoczko (\cite{Tymoczko}) proved that the Peterson variety admits a paving by affine cells, providing important insight into its topology. Subsequently, Harada, Horiguchi, and Masuda (\cite{HHM}) obtained explicit presentations of the equivariant cohomology rings of Peterson varieties in arbitrary Lie type. Further advances include equivariant Chevalley, Giambelli, and Monk-type formulae for Peterson Schubert classes, together with remarkable positivity phenomena that mirror many classical features of Schubert calculus while reflecting the distinctive geometry of Peterson varieties \cite{D, GMS, HT}.

A fundamental aspect of the geometry of the Peterson variety is its symmetry under torus actions. Unlike the full flag variety $G/B$, which is naturally acted upon by a maximal torus $T$, the Peterson variety is preserved by a distinguished one-dimensional subtorus $ \mathbb{G}_m \subset T$. For $G=\mathrm{GL}(n,\mathbb{C})$, this action is induced by the cocharacter
\[
\lambda(t)=\operatorname{diag}(t,t^2,\ldots,t^n).
\]
This torus action places the Peterson variety naturally within the framework of Geometric Invariant Theory, where one studies quotients of algebraic varieties by reductive group actions via linearized line bundles. In this setting, it is natural to ask how the stable and semistable loci depend on the choice of linearization, how the resulting GIT quotients vary across chambers in character space, and what geometry these quotients acquire. GIT quotients by one-dimensional tori have been studied in related
settings; in particular, \cite{BKP} investigates such quotients of
Grassmannians, including their semistable loci and geometric
properties.

The questions of normality and smoothness are particularly interesting
when the variety being quotiented is itself singular. Indeed, the
geometry of a GIT quotient need not reflect the singularities of the
original variety: examples in the setting of torus quotients of
Schubert varieties show that a singular Schubert variety can have a
smooth GIT quotient \cite{BP}. This phenomenon is especially relevant
for the Peterson variety, which is singular for \(n\ge3\) and
non-normal for \(n\ge4\). The stability condition may remove parts of
the singular or non-normal locus, while nontrivial stabilizers may
introduce new singularities in the quotient. Thus normality and
smoothness of the Peterson GIT quotient are genuine geometric
questions rather than formal consequences of the corresponding
properties of \(\mathrm{Pet}_n\).

In this paper, we undertake such a systematic study. Our main results establish precise criteria for the geometry of the GIT quotient
\[
X_\chi:= (\mathrm{Pet}_n)^{ss}_{G_0}(\mathcal L(\chi))
\sslash G_0,
\] where $G_0=\lambda(\mathbb G_m)$ and the linearization \(\mathcal L(\chi)\) is given by a 
regular dominant character \(\chi\) in the root lattice.

Our approach is built upon the Richardson stratification of \(\mathrm{Pet}_n\), which decomposes the variety into disjoint unions of strata
\[
Y^0_{I,J} := (B^-w_{J}B/B) \cap (Bw_{I}B/B) \cap \mathrm{Pet}_n,
\]
indexed by pairs of subsets \(J \subseteq I \subseteq S\), where \(S\) is the set of simple roots of type \(A_{n-1}\). We give an explicit description of the semistable and stable loci as disjoint unions of these strata. This determines the GIT chamber decomposition of the regular dominant cone, with wall-crossing phenomena governed by the vanishing of the Hilbert-Mumford weights at the \(\lambda\)-fixed points.

A particularly simple region of the character space is the
\emph{deep chamber}. Put
\[
K_r=S\setminus\{\alpha_r\},
\qquad 1\le r\le n-1,
\]
and define
\[
\mathcal C^\lambda_{\mathrm{deep}}
=
\left\{
\chi\in\mathcal D:
\langle w_{K_r}(\chi),\lambda\rangle<0
\text{ for }1\le r\le n-1
\right\},
\]
where \(\mathcal D\) is the regular dominant cone. Equivalently, in
this chamber
\[
\langle w_I(\chi),\lambda\rangle<0
\qquad
\text{for every }I\subsetneq S.
\]
Our first main result identifies the corresponding quotient
explicitly.

\medskip
\noindent\textbf{Theorem A} (see Theorem~ \ref{thm:deep-quotient}).
\emph{Let \(\chi\) be a regular dominant root-lattice character in
\(\mathcal C^\lambda_{\mathrm{deep}}\). Then
\[
(\mathrm{Pet}_n)^{ss}_{G_0}(\mathcal L(\chi))
=
(\mathrm{Pet}_n)^s_{G_0}(\mathcal L(\chi)),
\]
and
\[
(\mathrm{Pet}_n)^{ss}_{G_0}(\mathcal L(\chi))
\sslash G_0
\cong
\mathbb P(1,2,\ldots,n-1).
\]}
Thus the weighted projective
space \(\mathbb P(1,2,\ldots,n-1)\) provides a distinguished model
among all Peterson GIT quotients.

The same numerical description allows us to study variation of GIT.
The walls are the hyperplanes
\[
H_I
=
\left\{
\chi\in\mathcal D:
\langle w_I(\chi),\lambda\rangle=0
\right\}.
\]
Linearizations in the same chamber have the same semistable locus and
hence the same quotient, while crossing a wall gives the usual
projective birational VGIT morphisms. The stable locus is nonempty for every regular dominant
root-lattice character, and every quotient has dimension \(n-2\).
The VGIT comparison morphisms show that every quotient is birational to the deep-chamber quotient; in particular,
all of these quotients are rational. 

Our second main result concerns normality. Although
\(\mathrm{Pet}_n\) itself is non-normal in higher rank, the quotient
can still be normal. We first determine the normal locus of
\(\mathrm{Pet}_n\) explicitly. The interaction between this locus and the Hilbert-Mumford
inequalities leads to a complete criterion, with no genericity or
nonwall assumption on the linearization.

\noindent\textbf{Theorem B} (see Theorem~\ref{thm:normality-criterion}). 
\emph{Assume \(n\ge4\), and let \(\chi\) be a regular dominant
root-lattice character. Then
\[
(\mathrm{Pet}_n)^{ss}_{G_0}(\mathcal L(\chi))
\sslash G_0
\text{ is normal}
\quad\Longleftrightarrow\quad
\langle w_{K_r}(\chi),\lambda\rangle<0
\quad
(2\le r\le n-2).
\]}

Equivalently, the quotient is normal precisely when the semistable
locus is contained in the normal locus of \(\mathrm{Pet}_n\); see
Theorem~\ref{thm:normality-criterion} for details. In particular, the criterion
holds for all regular dominant root-lattice linearizations, including
those lying on GIT walls. The proof also treats wall linearizations. Using Toeplitz
coordinates on the Peterson big cell, we give an elementary
description of the maximal-parabolic strata \(Y^0_{S,K_r}\);
their \(G_0\)-orbits are indexed by cyclic classes of
\(r\)-element subsets of the \(n\)-th roots of unity. 
This auxiliary description is also useful for understanding the
local VGIT behavior along a generic wall. In particular, the
wall-crossing morphism from the adjacent normal chamber gives the
normalization in the generic single-wall case, and the above necklace
number records the cardinality of the fiber over the corresponding
fixed-point quotient.

As a consequence of Theorem~B, every Peterson GIT quotient is normal
for \(n\le5\), whereas for \(n\ge6\) the normality of the quotient
genuinely depends on the linearization.

Our third main result shows a striking contrast between normality and
smoothness. Whereas normality varies with the chamber, smoothness
depends only on \(n\).

\noindent\textbf{Theorem C} (see Theorem~\ref{thm:smoothness}). 
\emph{Let \(\chi\) be any regular dominant root-lattice character.
Then
\[
(\mathrm{Pet}_n)^{ss}_{G_0}(\mathcal L(\chi))
\sslash G_0
\]
is smooth if and only if $n\le3.$}

For \(n\ge4\), the proof exhibits stable points with nontrivial
finite stabilizer and applies Luna's \'etale slice theorem. Thus, unlike
normality, changing the chamber cannot remove the singularities of
the quotient.

Finally, we examine some consequences and examples of the chamber
structure. The deep-chamber quotient provides a common birational
model for all Peterson GIT quotients. We conclude with an explicit \(\mathrm{GL}(4,\mathbb C)\)-example showing that
wall crossing can preserve normality while changing the birational
geometry and destroying the weighted-projective-space description. This illustrates that the
variation of these quotients contains geometric information not
visible from normality or smoothness alone.

The paper is organized as follows. In Section~2 we fix the
root-theoretic notation, recall the Peterson-cell decomposition and
the distinguished one-dimensional torus action, and record the
Hilbert-Mumford formulas used throughout. In Section~3 we introduce
the Richardson strata, give the Bia{\l}ynicki-Birula decompositions, determine the stable, semistable, and
polystable loci, describe the GIT wall-and-chamber decomposition, and
identify the deep-chamber quotient, proving Theorem~A. In Section~4 we
determine the normal locus of the Peterson variety, analyze the
maximal-parabolic wall strata, and prove the normality criterion of
Theorem~B, including the normalization result for a generic single
wall. Section~5 proves the chamber-independent smoothness criterion of
Theorem~C. Finally, Section~6 presents an explicit wall crossing for
\(\mathrm{Pet}_4\) and records further properties of the
deep-chamber weighted projective quotient.

\section{Notation and Preliminaries}

We fix the notation and recall the basic facts about the Peterson
variety, its distinguished one-dimensional torus action, and
Geometric Invariant Theory that will be used throughout the paper.
For standard background on algebraic groups and root systems, we
refer to \cite{Hum,Hum2,Spr}.
Let $G = \mathrm{GL}(n,\mathbb{C})$ be the general linear group. Denote by $B$ the Borel subgroup of upper triangular matrices, and by $T$ the maximal torus of diagonal matrices in \(\mathrm{GL}(n,\mathbb{C})\). The flag variety $\mathrm{Flags}(\mathbb{C}^n)$ is identified with $G/B$.

Denote by \(X(T)\) the character group of \(T\) and by \(Y(T)\) the cocharacter group. Set
\[
E_1 = X(T) \otimes \mathbb{R}, \qquad E_2 = Y(T) \otimes \mathbb{R},
\]
with the canonical pairing \(\langle \cdot,\cdot \rangle : E_1 \times E_2 \to \mathbb{R}\). Let $R = \{\epsilon_i - \epsilon_j : 1 \le i, j \le n,\ i \neq j\}$ be the set of roots with respect to $T$, and let
\[
R^+ = \{\epsilon_i - \epsilon_j : 1 \le i < j \le n\}
\]
be the set of positive roots with respect to $(T,B)$, where \(\epsilon_i\) denotes the \(i\)-th coordinate character of \(T\).
Let $S = \{\alpha_i := \epsilon_i - \epsilon_{i+1} : i = 1,\dots,n-1\} \subset R^+$ be the set of simple roots.
The Weyl group $W$ is the symmetric group $S_n$, acting by permutation of coordinates. For $\beta=\sum_{i}c_{i}\alpha_{i}\in R^{+}$ we write $\mathrm{ht}(\beta)=\sum_{i}c_{i}$ for the
height of $\beta$.

Let \[ Q:=\mathbb ZS\subset X(T) \] be the root lattice and set \[ Q_{\mathbb R}:=Q\otimes_{\mathbb Z}\mathbb R = \left\{ \sum_{i=1}^n x_i\epsilon_i: \sum_{i=1}^n x_i=0 \right\}. \] For \(1\le j\le n-1\), let \[ \varpi_j = \epsilon_1+\cdots+\epsilon_j -\frac{j}{n}\sum_{i=1}^n\epsilon_i \in Q_{\mathbb R} \] be the \(j\)-th fundamental weight of the root system \(A_{n-1}\). Thus $ \langle\varpi_i,\alpha_j^\vee\rangle=\delta_{ij},$ where $\check{\alpha}_j(t) = \operatorname{diag}(1,\dots,1,t,t^{-1},1,\dots,1)$ is the simple coroot (with $t$ in the $j$-th position). Every regular dominant element of \(Q_{\mathbb R}\) has a unique expression \[ \chi=\sum_{j=1}^{n-1}m_j\varpi_j, \qquad m_j>0, \] and for a root-lattice character \(\chi\in Q\) one has $m_j=\langle\chi,\alpha_j^\vee\rangle\in\mathbb Z_{>0}.$


Define the dominant chamber in \(E_2\):
\[
\overline{C(B)} = \{ \nu \in E_2 \mid \langle \alpha, \nu \rangle \ge 0 \text{ for all } \alpha \in R^+ \}.
\]

For each root \(\alpha \in R\), there is an injective homomorphism \(\phi_\alpha : \mathrm{SL}_2(\mathbb{C}) \to \mathrm{GL}(n,\mathbb{C})\) such that
\[
\check{\alpha}(t) = \phi_\alpha\!\left( \begin{pmatrix} t & 0 \\ 0 & t^{-1} \end{pmatrix} \right).
\]
The simple reflection \(s_i = s_{\alpha_i}\) acts on characters by
\[
s_i(\chi) = \chi - \langle \chi, \check{\alpha}_i \rangle \alpha_i.
\]

Let \(U\) be the unipotent radical of the Borel subgroup \(B\). For each root \(\alpha \in R\), denote by \(U_\alpha\) the one‑dimensional \(T\)-stable root subgroup of \(G\) corresponding to \(\alpha\). There is an isomorphism \(u_\alpha : \mathbb{C} \to U_\alpha\) satisfying
\[
t \, u_\alpha(a) \, t^{-1} = u_\alpha( \alpha(t) a ) \qquad \text{for all } t \in T,\ a \in \mathbb{C}.
\]

There is a partial order on \(X(T)\) defined by \(\psi \le \chi\) if and only if \(\chi - \psi\) is a non‑negative integral linear combination of the simple roots.

For a subset $I \subseteq S$, let $W_I$ be the parabolic subgroup generated by the simple reflections $s_{\alpha_i}$ with $\alpha_i \in I$. The longest element of $W_I$ is denoted by $w_{I}$. Explicitly, if $I$ corresponds to a decomposition $n = i_1 + \dots + i_k$ into block sizes, then $w_{I}$ is the permutation
\[
w_{I} = (i_1,i_1-1,\dots,1,\; i_1+i_2,\dots,i_1+1,\; \dots,\; n,\dots,i_1+\cdots+i_{k-1}+1).
\] The Bruhat order on these elements satisfies
$$
w_J \leq w_{J'} \quad \text{if and only if} \quad J \subseteq J'.
$$

\subsection{The Peterson variety}
The complete flag variety $Fl_n = G / B$
parameterizes complete flags in $\mathbb{C}^n$; that is, nested sequences of subspaces
$$
0 \subset V_1 \subset V_2 \subset \cdots \subset V_{n-1} \subset V_n = \mathbb{C}^n
$$
with $\dim V_i = i$ for each $i = 1, \dots, n$. We next recall the realization of the type \(A\) Peterson variety as a
regular nilpotent Hessenberg variety. 

Let $N \in \mathfrak{gl}(n,\mathbb{C})$ be a \emph{regular nilpotent matrix}, viewed as a linear map  
$N : \mathbb{C}^n \to \mathbb{C}^n$. By definition, $N$ satisfies
$$
N^k \neq 0 \quad \text{for } 1 \le k \le n-1, 
\quad \text{and} \quad N^{n} = 0.
$$
Up to a change of basis, we may take $N$ to be the \emph{standard Jordan block} with 1’s along the superdiagonal and 0’s elsewhere:
$$
N = 
\begin{pmatrix}
0 & 1 & 0 & \cdots & 0 \\
0 & 0 & 1 & \cdots & 0 \\
\vdots & \vdots & \ddots & \ddots & \vdots \\
0 & 0 & \cdots & 0 & 1 \\
0 & 0 & \cdots & 0 & 0
\end{pmatrix}.
$$
This choice ensures that $N$ is regular, i.e., its Jordan form consists of a single block of size $n$. We fix this regular nilpotent operator throughout this paper.

The \emph{Peterson variety}, denoted by $\mathrm{Pet}_n$, is a closed subvariety of the flag variety $Fl_n$ defined by
$$
\mathrm{Pet}_n := \{ (V_i) \in Fl_n \mid N V_i \subseteq V_{i+1} \text{ for all } 1 \le i \le n-1 \}.
$$
Thus, a flag $(V_1, \dots, V_n)$ lies in $\mathrm{Pet}_n$ if and only if the action of the nilpotent operator $N$ maps each subspace $V_i$ into the $V_{i+1}$ in the flag.

Equivalently, $\mathrm{Pet}_n$ can be described as a \emph{Hessenberg variety} associated with the Hessenberg function
$$
h(i) = i + 1, \quad 1 \le i \le n-1 \,\, and \,\, h(n)=n.
$$
We recall that for a linear operator $X \in \mathfrak{gl}(n,\mathbb{C})$ and a Hessenberg function 
$h : \{1, \dots, n\} \to \{1, \dots, n\}$, the corresponding Hessenberg variety is
$$
Hess(X, h) = \{ (V_i) \in Fl_n \mid X V_i \subseteq V_{h(i)} \}.
$$
Therefore,
$$
\mathrm{Pet}_n = Hess(N, h), \quad \text{where } h(i) = i+1.
$$

There is also a description of $\mathrm{Pet}_n$ in terms of the Lie algebra $\mathfrak{g} = \mathfrak{gl}(n,\mathbb{C})$.  
Let $\mathfrak{b}$ denote the Lie algebra of $B$, and for each simple root $\alpha_j$, let $\mathfrak{g}_{-\alpha_j}$ denote the corresponding negative root space.

Define
$$
H_{[n-1]} := \mathfrak{b} \oplus \bigoplus_{j \in [n-1]} \mathfrak{g}_{-\alpha_j}.
$$
Then the Peterson variety can be described as the following subvariety of $G / B$:
$$
\mathrm{Pet}_n = \{ gB \in G/B \mid g^{-1} N g \in H_{[n-1]} \}.
$$
In this formulation, $\mathrm{Pet}_n$ consists of those flags $gB$ such that when $N$ is conjugated by $g^{-1}$, the result lies in the specified subspace $H_{[n-1]}$ of the Lie algebra.

The Peterson variety $\mathrm{Pet}_n$ is a \emph{projective}, \emph{irreducible} algebraic variety.  
Its dimension is
$$
\dim(Pet_n) = n - 1.
$$

\subsection{The distinguished one-dimensional torus action}
The natural action of $G$ on the flag variety $Fl_n=G/B$ restricts to an action of $T$, and the set of $T$-fixed points in $Fl_n$ is in bijection with $S_n$, where each permutation $w \in S_n$ corresponds to the fixed point $wB \in G/B$.

Note that the condition $NV_i \subseteq V_{i+1}$ is not preserved by the $T$ action on the flag variety and as a consequence, the Peterson variety is not preserved by the full maximal torus.
It is, however, preserved by a distinguished one-dimensional subtorus.  We fix this action and the
associated cocharacter, which we will use throughout the paper. Let \[ \lambda:\mathbb G_m\longrightarrow T \subset G, \qquad \lambda(t)=\diag(t,t^2,\ldots,t^n), \] be the one-parameter subgroup of $G$ and set $G_0:=\lambda(\mathbb G_m).$ For every simple root $\alpha_i$, we have $\langle\alpha_i,\lambda\rangle=-1,$ and therefore \[ \langle\beta,\lambda\rangle=-\operatorname{ht}(\beta) \qquad(\beta\in R^+). \] In particular, \(-\lambda\) is strictly dominant. For the standard regular nilpotent $N=\sum_{i=1}^{n-1}E_{i,i+1},$ we have \[ \lambda(t)N\lambda(t)^{-1}=t^{-1}N. \] Hence the action of \(G_0\) on \(G/B\) preserves $\mathrm{Pet}_n.$ The \(G_0\)-fixed points of \(\mathrm{Pet}_n\) are \[ p_I:=w_IB/B, \qquad I\subseteq S. \] 

\subsection{Peterson cells and their affine paving}\label{affine-paving} The Schubert-cell decomposition of \(G/B\) induces an affine paving of
the Peterson variety. We recall this paving and the corresponding
closure relations.
For \(w\in W\), let
\[
C_w:=BwB/B,
\qquad
C^w:=B^-wB/B
\]
denote the Schubert and opposite Schubert cells in \(G/B\), respectively,
and let $\overline{C_w}$ and 
$\overline{C^w}$
denote the corresponding Schubert and opposite Schubert varieties, respectively. 

A Schubert cell \(C_w\) meets \(\mathrm{Pet}_n\) nontrivially precisely
when \(w=w_I\) for some subset \(I\subseteq S\); the same statement holds
for the opposite Schubert cells \(C^w\); see
\cite[Lemma~3.5]{Abe}. Accordingly, for \(I\subseteq S\), define the
\emph{Peterson cell}
\[
X_I^0
:=
C_{w_I}\cap\mathrm{Pet}_n
=
(Bw_IB/B)\cap\mathrm{Pet}_n,
\]
and its closure
\[
X_I:=\overline{X_I^0}.
\]
Similarly, define the opposite Peterson cell
\[
\Omega_I^0
:=
C^{w_I}\cap\mathrm{Pet}_n
=
(B^-w_IB/B)\cap\mathrm{Pet}_n,
\]
and set $\Omega_I:=\overline{\Omega_I^0}.$ Since $G/B=\bigsqcup_{w\in W}C_w,$
the above nonemptiness criterion gives the disjoint decomposition
\[
\mathrm{Pet}_n
=
\bigsqcup_{I\subseteq S}X_I^0.
\]
Tymoczko's affine paving of regular nilpotent Hessenberg varieties
\cite{Tymoczko1} implies that each Peterson cell is an affine space:
\[
X_I^0\simeq\mathbb A^{|I|}.
\]
In particular, $\dim X_I^0=|I|$ and $\dim\mathrm{Pet}_n=n-1,$
the latter being attained for \(I=S\).

The closures of the Peterson cells are compatible with inclusion of the
indexing subsets. More precisely,
\begin{equation}\label{eq:peterson-cell-closure}
X_I
=
\overline{X_I^0}
=
\bigsqcup_{J\subseteq I}X_J^0;
\end{equation}
see \cite[Equation~(3.7)]{Abe}. Consequently, $X_J\subseteq X_I$ if and only if $J\subseteq I.$
Equivalently, the Bruhat order on the parabolic longest elements satisfies
\[
w_J\le w_I
\qquad\Longleftrightarrow\qquad
J\subseteq I.
\]

We shall also use the corresponding Peterson--Schubert and opposite
Peterson--Schubert intersections. They satisfy
\[
X_I\cap\Omega_J\neq\varnothing
\qquad\Longleftrightarrow\qquad
J\subseteq I,
\]
and, in particular,
\[
X_I\cap\Omega_I=\{p_I\},
\qquad
p_I:=w_IB/B.
\]
These cells and their intersections will provide the geometric
stratification used below to compute the Hilbert--Mumford numerical
function and hence the stable and semistable loci.

\subsection{Geometric Invariant Theory}
We begin by recalling some basic results from Geometric Invariant Theory (GIT); the standard reference is \cite{GIT}. Let $X$ be a projective variety equipped with an ample line bundle $\mathcal L$, and let $G$ be a reductive group acting on the polarized pair $(X,\mathcal L)$. The general theory of GIT shows that the quotient is again a projective variety. The geometric meaning of the quotient is reflected in the notions of semistability, stability, and polystability of orbits.

\begin{definition}
Let $G$ be a reductive algebraic group acting on a projective variety $X$, and let
$\mathcal L$ be a $G$-linearized ample line bundle on $X$.

\begin{enumerate}
    \item A point $x\in X$ is called \emph{semistable} (with respect to $\mathcal L$) if there exists
    a $G$-invariant section $s\in H^0(X,\mathcal L^{\otimes m})^G$ for some $m>0$ such that
    $s(x)\neq 0$. The set of semistable points is denoted by $X_G^{ss}(\mathcal L)$.

    \item A point $x\in X_G^{ss}(\mathcal L)$ is called \emph{stable} if its $G$-orbit is closed in
    $X_G^{ss}(\mathcal L)$ and its stabilizer in $G$ is finite. The set of stable points is denoted by
    $X_G^s(\mathcal L)$.

    \item A point $x\in X_G^{ss}(\mathcal L)$ is called \emph{polystable} if its $G$-orbit is closed in
    $X_G^{ss}(\mathcal L)$. Equivalently, $x$ is polystable if it represents a closed orbit in the GIT
    quotient. The set of polystable points is denoted by $X_G^{ps}(\mathcal L)$.

    \item A point $x \in X$ is called unstable with respect to $\mathcal L$ if $x$ is not semistable.
\end{enumerate}
\end{definition}

While stability implies polystability, and polystability implies semistability, it is specifically the polystable orbits that are parameterized by the GIT quotient. This is because every semistable point in the variety contains a unique polystable orbit within the closure of its $G$-orbit relative to the semistable locus.  As we see in the following theorem, the semistable locus is an open subset of 
$X$ over which a good quotient exists.

\begin{theorem}[Mumford \cite{GIT}]
Let $G$ be a reductive group acting on a projective variety $X$, and let
$\mathcal L$ be a $G$-linearized ample line bundle on $X$. Then the semistable
locus $X_G^{ss}(\mathcal L)$ admits a good quotient
\[
\pi : X_G^{ss}(\mathcal L)\longrightarrow X_G^{ss}(\mathcal L)//G.
\]
Moreover, there exists an open subset
\[
U \subseteq X_G^{ss}(\mathcal L)//G
\]
such that
\[
X_G^s(\mathcal L)=\pi^{-1}(U),
\]
and the restriction
\[
\pi|_{X_G^s(\mathcal L)} : X_G^s(\mathcal L)\longrightarrow U
\]
is a geometric quotient.

In addition, the quotient is projective, and there is a canonical isomorphism
\[
X_G^{ss}(\mathcal L)//G \;\cong\; \operatorname{Proj}(R^G),
\]
where
\[
R=\bigoplus_{k\ge 0} H^0\!\left(X,\mathcal L^{\otimes k}\right)
\]
is the section ring of $\mathcal L$ and $R^G$ denotes its subring of $G$-invariant
elements.
\end{theorem}

So the GIT quotient $X_G^{ss}(\mathcal L)//G$ exists as a projective variety and
parametrizes the closed orbits in the semistable locus, i.e.\ the
polystable orbits. In particular, the invariant section ring
\[
\bigoplus_{k\ge 0} H^0\!\left(X, \mathcal L^{\otimes k}\right)^G
\]
is finitely generated. Moreover, the GIT quotient preserves the usual separation properties expected of a projective variety. Unlike an ordinary topological quotient, which may
be non-Hausdorff, the GIT quotient identifies semistable points whose
orbit closures meet inside $X_G^{ss}(\mathcal L)$. Two semistable points
$x,y \in X_G^{ss}(\mathcal L)$ are said to be S-equivalent if the closures of their
$G$-orbits intersect in $X_G^{ss}(\mathcal L)$, equivalently if they have the same
associated closed (polystable) orbit.

Locally, GIT quotients admit an affine description. Let $x \in X$ be a semistable point, so that there exists a $G$-invariant section 
\[
s \in H^0(X, \mathcal L^{\otimes k})^G
\]
for some $k > 0$ with $s(x) \neq 0$. Since $\mathcal L$ is ample, the open subset
\[
X_s := X \setminus V(s)
\]
is affine. The corresponding affine GIT quotient is given by
\[
X_s /\!/ G \;:=\; \operatorname{Spec}\big( \mathbb{C}[X_s]^G \big).
\]
These affine quotients form an open cover of the global GIT quotient \[ X_G^{ss}(\mathcal L)\sslash G. \] For a fixed \(G\)-invariant affine open subset \(U\subseteq X\), the affine quotient \[ U\sslash G=\operatorname{Spec}\mathbb C[U]^G \] depends only on the \(G\)-action on \(U\). The collection of invariant affine opens occurring in the GIT construction, and hence the global semistable locus and quotient, may nevertheless depend on the chosen linearization.

\subsection{} For \(\chi\in X(T)\), let \(\mathbb C_\chi\) denote the corresponding one-dimensional \(B\)-module and set \[ \mathcal L(\chi) = G\times^B\mathbb C_{-\chi}\longrightarrow G/B. \] With this convention, \[ H^0(G/B,\mathcal L(\chi))\simeq V(\chi)^* \] for dominant \(\chi\), and \(\mathcal L(\chi)\) is ample when \(\chi\) is regular dominant. We use the same notation for its restriction to \(\mathrm{Pet}_n\). The natural \(G\)-linearization of \(\mathcal L(\chi)\) restricts to a \(G_0=\lambda(\mathbb G_m)\)-linearization. Unless explicitly stated otherwise, all stability notions below refer to this restricted linearization.

The action of \(G_0\) on \(\mathrm{Pet}_n\subseteq G/B\) is the
restriction of the natural left action of \(G\) on \(G/B\).

For a regular dominant root-lattice character \(\chi\), we denote the
semistable, stable, and polystable loci with respect to the induced
\(G_0\)-linearization on \(\mathcal L(\chi)\) by
\[
(\mathrm{Pet}_n)^{ss}_{G_0}(\mathcal L(\chi)),
\qquad
(\mathrm{Pet}_n)^{s}_{G_0}(\mathcal L(\chi)),
\qquad
(\mathrm{Pet}_n)^{ps}_{G_0}(\mathcal L(\chi)),
\]
respectively. The associated GIT quotient will be denoted by
\[
(\mathrm{Pet}_n)^{ss}_{G_0}(\mathcal L(\chi))
\sslash G_0.
\]
Unless otherwise stated, all notions of stability and all GIT quotients
below are taken with respect to this \(G_0\)-linearization.

\subsection{The Hilbert-Mumford Criterion}

Let $H$ be a reductive algebraic group acting morphically on a projective variety $X$, and let $\mathcal{L}$ be an $H$-linearized very ample line bundle on $X$.

To evaluate stability numerically, let $x \in \mathbb{P}(H^0(X, \mathcal{L})^*)$ and let $\hat{x}$ be a representative point in the affine cone $\hat{X}$ over $X$ lying above $x$. Let $\lambda' \colon \mathbb{G}_m \to H$ be a one-parameter subgroup. Since the image $\lambda'(\mathbb{G}_m)$ is a torus, we can choose a basis $\{v_1, v_2, \ldots, v_k\}$ of $H^0(X, \mathcal{L})^*$ consisting of weight vectors, such that there exist integers $m_1, m_2, \ldots, m_k$ satisfying
\begin{align*}
    \lambda'(t) \cdot v_i = t^{m_i} v_i \quad \text{for all } 1 \leq i \leq k \text{ and } t \in \mathbb{G}_m.
\end{align*}
Writing $\hat{x} = \sum_{i=1}^k c_i v_i$ with coefficients $c_i \in \mathbb{C}$, the \textit{Hilbert-Mumford numerical function} is defined as
\begin{align*}
    \mu^{\mathcal{L}}(x, \lambda') := -\min_i \{ m_i : c_i \neq 0 \}.
\end{align*}

With this numerical setup, we recall the Hilbert-Mumford criterion for stability.

\begin{theorem}[{See \cite[Theorem 2.1]{GIT}}]\label{HM theorem for G}
    Let $x \in X$. Then:
    \begin{enumerate}
        \item $x \in X^{ss}_H(\mathcal{L})$ if and only if $\mu^{\mathcal{L}}(x, \lambda') \geq 0$ for all one-parameter subgroups $\lambda'$ of $H$.
        \item $x \in X^s_H(\mathcal{L})$ if and only if $\mu^{\mathcal{L}}(x, \lambda') > 0$ for all non-trivial one-parameter subgroups $\lambda'$ of $H$.
    \end{enumerate}
\end{theorem}

Applying this locally to the action of a single one-parameter subgroup, we obtain the following corollary.

\begin{corollary}\label{HM-corollary}
    Let $H, X,$ and $\mathcal{L}$ be as above, and let $\lambda' \colon \mathbb{G}_m \to H$ be a one-parameter subgroup. Then:
    \begin{enumerate}
        \item $x \in X^{ss}_{\lambda'(\mathbb{G}_m)}(\mathcal{L})$ if and only if both $\mu^{\mathcal{L}}(x, \lambda')$ and $\mu^{\mathcal{L}}(x, -\lambda')$ are non-negative.
        \item $x \in X^s_{\lambda'(\mathbb{G}_m)}(\mathcal{L})$ if and only if both $\mu^{\mathcal{L}}(x, \lambda')$ and $\mu^{\mathcal{L}}(x, -\lambda')$ are strictly positive.
    \end{enumerate}
\end{corollary}

Let $H, B, T$ and $\overline{C(B)}$ be as above and $H/B$ be the flag variety. The numerical function is explicitly computable on flag varieties, as demonstrated by Seshadri.

\begin{lemma}[{See \cite[Lemma 5.1]{Ses2}}]\label{CSS}
    Let $x = bwB/B$ for some $b \in B$ and $w \in W$. If $\lambda' \in \overline{C(B)}$ is a dominant one-parameter subgroup, then
    \begin{align*}
        \mu^{\mathcal{L}(\chi)}(x, \lambda') = -\langle w(\chi), \lambda' \rangle.
    \end{align*}
    \textit{(Note: The negative sign arises here because we are considering the left action of $B$ on $H/B$, whereas \cite[Lemma 5.1]{Ses2} uses the right action.)}
\end{lemma}

By imitating the proof of Lemma ~ \ref{CSS} for the opposite Borel subgroup $B^-$, one naturally deduces the following variation.

\begin{lemma}\label{variation of Seshadri's lemma}
    Let $w \in W$ and $x \in B^-wB/B$. For every dominant one-parameter subgroup $\lambda' \in \overline{C(B)}$, we have
    \begin{align*}
        \mu^{\mathcal{L}(\chi)}(x, -\lambda') = \langle w(\chi), \lambda' \rangle.
    \end{align*}
\end{lemma}

\begin{lemma}\label{Section2:Lemma 2.1.4}Let  $\chi$ be regular dominant character of $T$. Then
	\begin{enumerate}
		\item  if $x\in BwB/B$, we have $\mu^{\mathcal{L}(\chi)}(x,-\lambda)=\langle w(\chi),\lambda\rangle$
		\item if $x\in B^{-}wB/B$, we have $\mu^{\mathcal{L}(\chi)}(x,\lambda)=-\langle w(\chi),\lambda\rangle$.
	\end{enumerate}
\end{lemma}
\begin{proof} We have $-\lambda\in\overline{C(B)}$. 
	
	{\it{(1)}}: Let $x\in BwB/B$. 
From Lemma ~ \ref{CSS}, we have
\begin{align*}
	\mu^{\mathcal{L}(\chi)}(x,-\lambda)=-\langle w(\chi),-\lambda\rangle=\langle w(\chi),\lambda\rangle.
\end{align*}
{\it{(2)}}: Let $x\in B^{-}wB/B$. From Lemma ~ \ref{variation of Seshadri's lemma}, we have 
\begin{align*}
	\mu^{\mathcal{L}(\chi)}(x,\lambda)=\mu^{\mathcal{L}(\chi)}(x,-(-\lambda))=\langle w(\chi),-\lambda\rangle=-\langle w(\chi),\lambda\rangle.
\end{align*}
\end{proof}

\section{The GIT quotient of Peterson variety}
In this section we determine the semistable, stable, and polystable
loci for the \(G_0\)-action on \(\mathrm{Pet}_n\), describe the
resulting chamber structure, and identify the quotient in a
distinguished chamber. The key geometric input is the Richardson stratification together with the
\(\lambda\)-limits of its strata.

\subsection{Richardson stratification for $\mathrm{Pet}_n$} To apply the Hilbert-Mumford criterion, we refine the Peterson-cell
decomposition by intersecting Schubert and opposite Schubert cells.
The resulting locally closed pieces have constant numerical weights.

Recall from \ref{affine-paving} that, for every $J\subseteq S$, we have \begin{align}
	X_{J}^{0}:&= (Bw_{J}B/B)\cap \mathrm{Pet}_{n} \text{ and } X_{J}:=\overline{X_{J}^{0}}\\
	\Omega_{J}^{0}:&=(B^{-}w_{J}B/B)\cap \mathrm{Pet}_{n}\text{ and }\Omega_{J}:=\overline{\Omega_{J}^{0}}
\end{align}
For every $I,J\subseteq S$, we define \begin{align}
	Y^{0}_{I,J}:&=(B^{-}w_{J}B/B)\cap (Bw_{I}B/B)\cap \mathrm{Pet}_{n}\\
	Y_{I,J}:&=\overline{Y^{0}_{I,J}}
\end{align}

From \cite[Lemma 3.5]{Abe}, the Peterson variety admits a decomposition into intersections with opposite Schubert cells:
\begin{align*}
    \mathrm{Pet}_{n} = \bigsqcup_{J\subseteq S}\bigsqcup_{I\subseteq S} Y^{0}_{I,J},
\end{align*}
where $Y^{0}_{I,J} := (B^{-}w_{J}B/B) \cap (Bw_{I}B/B) \cap \mathrm{Pet}_{n}$. We know that the intersection $(Bw_{I}B/B) \cap (B^{-}w_{J}B/B) \neq \emptyset$ if and only if $w_{J} \leq w_{I}$ in the Bruhat order. For the longest elements of parabolic subgroups, this Bruhat inequality is equivalent to $J \subseteq I$. Therefore, the decomposition simplifies to:
\begin{align}\label{Section 3: Eq 3.1.1}
    \mathrm{Pet}_{n} = \bigsqcup_{J \subseteq I \subseteq S} Y^{0}_{I,J}.
\end{align}

\begin{remark} The irreducibility of Peterson-Richardson intersections is not known in general; see \cite[Question~7.6(i)]{Horiguchi}. In fact, such intersections need not always be irreducible. The particular strata used in this paper, however, are cases for which irreducibility follows directly: namely \(Y^0_{I,\emptyset}\), \(Y^0_{I,I}\), and the endpoint strata \(Y^0_{S,S\setminus\{\alpha_1\}}\) and \(Y^0_{S,S\setminus\{\alpha_{n-1}\}}\). \end{remark}

\subsection{Bialynicki--Birula decomposition for $\mathrm{Pet}_n$} Recall that the
$G_0=\lambda(\mathbb G_m)$-fixed points of $\mathrm{Pet}_n$ are the points
\[
p_I := w_{I}B/B, \qquad I \subseteq S.
\]

For each $I \subseteq S$, define the attracting and repelling sets
\[
W_I^+ := \{x \in \mathrm{Pet}_n \mid \lim_{t \to 0} \lambda(t)\cdot x = p_I\},
\qquad
W_I^- := \{x \in \mathrm{Pet}_n \mid \lim_{t \to \infty} \lambda(t)\cdot x = p_I\}.
\]

\begin{proposition}\label{thm:BB-Peterson}
The Peterson variety admits the disjoint decompositions
\[
\mathrm{Pet}_n
=
\bigsqcup_{I\subseteq S}W_I^+
\qquad\text{and}\qquad
\mathrm{Pet}_n
=
\bigsqcup_{I\subseteq S}W_I^-.
\]
Each \(W_I^\pm\) is a locally closed subvariety of
\(\mathrm{Pet}_n\). Moreover, these decompositions are compatible with
the Peterson--Richardson stratification:
\[
W_I^+
=
\bigsqcup_{J\supseteq I}Y^0_{J,I}
=
\Omega_I^0,
\qquad
W_I^-
=
\bigsqcup_{J\subseteq I}Y^0_{I,J}
=
X_I^0
\simeq
\mathbb A^{|I|}.
\]

If \(\mathrm{Pet}_n\) is smooth at \(p_I\), then \(W_I^+\) is an
affine space. More precisely, if
\[
T_{p_I}\mathrm{Pet}_n
=
T_{p_I}^+\oplus T_{p_I}^-
\]
denotes the decomposition into positive and negative
\(\lambda\)-weight spaces, then
\[
W_I^+
\simeq
\mathbb A^{\dim T_{p_I}^+},
\qquad
W_I^-
\simeq
\mathbb A^{\dim T_{p_I}^-}.
\]
Consequently, $\dim W_I^\pm
=
\#\{
\text{positive/negative \(\lambda\)-weights on }
T_{p_I}\mathrm{Pet}_n
\},$
with weights counted with multiplicity. Under the hypothesis that \(\mathrm{Pet}_n\) is smooth at \(p_I\), we
therefore have
\[
\dim T_{p_I}^- = |I|,
\qquad
\dim T_{p_I}^+ = n-1-|I|.
\] 
\end{proposition}

\begin{proof}
Since $\mathrm{Pet}_n$ is projective and $\lambda$ acts algebraically on it, the limit
\[
\lim_{t \to 0} \lambda(t)\cdot x
\]
exists for every $x \in \mathrm{Pet}_n$ by properness. The limit point is necessarily
fixed by $\lambda$, because
\[
\lambda(s)\cdot \Bigl(\lim_{t \to 0} \lambda(t)\cdot x\Bigr)
=
\lim_{t \to 0} \lambda(st)\cdot x
=
\lim_{t \to 0} \lambda(t)\cdot x
\]
for every $s \in \mathbb G_m$.

By the description of the fixed locus, this limit must be one of the points
$p_I = w_{I}B/B$. Hence every point of $\mathrm{Pet}_n$ belongs to at least one
attracting set $W_I^+$.

The sets $W_I^+$ are pairwise disjoint: if
$x \in W_I^+ \cap W_J^+$, then the limit $\lim_{t \to 0}\lambda(t)\cdot x$
is both $p_I$ and $p_J$, so $p_I = p_J$ and hence $I = J$. Therefore
\[
\mathrm{Pet}_n = \bigsqcup_{I \subseteq S} W_I^+.
\]
The same argument applied to $t \to \infty$ gives
\[
\mathrm{Pet}_n = \bigsqcup_{I \subseteq S} W_I^-.
\]

Each $W_I^\pm$ is locally closed by standard properties of attracting sets for
an algebraic $\mathbb G_m$-action. 

We already have that the  Peterson variety admits the Richardson stratification
\[
\mathrm{Pet}_n = \bigsqcup_{J \subseteq I \subseteq S} Y^0_{I,J}.
\]
For a point $x \in Y^0_{I,J}$, since $x$ lies in the opposite Schubert cell $B^-w_{J}B/B$ and $\lambda$ is antidominant, the limit as $t\to 0$ is $p_J = w_{J}B/B$. Again since $x$ lies in the Schubert cell $Bw_{I}B/B$, the limit as $t\to\infty$ is $p_I = w_{I}B/B$. Hence $Y^0_{I,J} \subseteq W_J^+$ and $Y^0_{I,J} \subseteq W_I^-$. Consequently,
\[
W_I^+ = \bigsqcup_{J \supseteq I} Y^0_{J,I}, \qquad
W_I^- = \bigsqcup_{K \subseteq I} Y^0_{I,K},
\]

proving the compatibility of the decomposition with the Richardson stratification.

By Tymoczko's affine paving, each Peterson cell $
X_I^0 = (B w_I B/B) \cap \mathrm{Pet}_n$
is isomorphic to an affine space of dimension $|I|$ (see \cite{Tymoczko}, \cite{Tymoczko1}). Therefore
\[
W_I^- \cong \mathbb{A}^{|I|}
\,\, \text{for every} \,\, I \subseteq S.\]

Suppose \(\mathrm{Pet}_n\) is smooth at \(p_I\). The locus
\(\operatorname{Sing}(\mathrm{Pet}_n)\) is closed, and it is \(G_0\)-stable because \(G_0\) acts by
automorphisms of \(\mathrm{Pet}_n\). Let \(x\in W_I^{+}\) and suppose \(x\) were singular. Then
\(G_0\cdot x\subseteq\operatorname{Sing}(\mathrm{Pet}_n)\), and since the latter is closed,
\[
\overline{G_0\cdot x}\subseteq\operatorname{Sing}(\mathrm{Pet}_n).
\]
The morphism \(\mathbb G_m\to\mathrm{Pet}_n\), \(t\mapsto\lambda(t)\cdot x\), extends to
\(\mu:\mathbb A^1\to\mathrm{Pet}_n\) with \(\mu(0)=p_I\); as \(\mathbb G_m\) is dense in
\(\mathbb A^1\), we get \(p_I\in\overline{G_0\cdot x}\). Hence
\(p_I\in\operatorname{Sing}(\mathrm{Pet}_n)\), a contradiction. Therefore
\(W_I^{+}\subseteq\mathrm{Pet}_n^{\mathrm{sm}}\), and the same argument with \(t\to\infty\) gives
\(W_I^{-}\subseteq\mathrm{Pet}_n^{\mathrm{sm}}\).

The smooth locus \(\mathrm{Pet}_n^{\mathrm{sm}}\) is a smooth
quasi-projective \(G_0\)-variety. Since it is normal and hence we get a cover of \(\mathrm{Pet}_n^{\mathrm{sm}}\) by
\(G_0\)-invariant affine open subsets. Hence the hypotheses of the
Bia{\l}ynicki--Birula theorem
\cite[Theorem~4.1]{BB}
are satisfied.

Since \(p_I\) is an isolated \(G_0\)-fixed point, the attracting set of
\(p_I\) in \(\mathrm{Pet}_n^{\mathrm{sm}}\) is an affine space whose
dimension is the number of positive \(\lambda\)-weights on
\(T_{p_I}\mathrm{Pet}_n\). Since
\(W_I^+\subseteq\mathrm{Pet}_n^{\mathrm{sm}}\), this attracting set is
precisely \(W_I^+\). Therefore $W_I^+
\simeq
\mathbb A^{\dim T_{p_I}^+}.$

Applying the same argument to the inverse one-parameter subgroup
\(t\mapsto\lambda(t^{-1})\), whose attracting set is \(W_I^-\), gives
\[
W_I^-
\simeq
\mathbb A^{\dim T_{p_I}^-}.
\]

On the other hand, we already know that $W_I^-=X_I^0\simeq\mathbb A^{|I|}.$ Hence $\dim T_{p_I}^-=|I|.$
Since \(p_I\) is smooth and $\dim\mathrm{Pet}_n=n-1,$
we have $\dim T_{p_I}\mathrm{Pet}_n=n-1.$
Moreover, \(p_I\) is an isolated fixed point, so the zero-weight
subspace of \(T_{p_I}\mathrm{Pet}_n\) is zero. Thus
\[
T_{p_I}\mathrm{Pet}_n
=
T_{p_I}^+\oplus T_{p_I}^-,
\]
and consequently
\[
\dim T_{p_I}^+
=
n-1-|I|.
\]
Therefore 
\[
W_I^+
\simeq\mathbb A^{\,n-1-|I|},
\qquad
W_I^-
\simeq\mathbb A^{|I|},
\]
and their dimensions are respectively the numbers of positive and
negative \(\lambda\)-weights on
\(T_{p_I}\mathrm{Pet}_n\), counted with multiplicity.
\end{proof}

\subsection{Semistable, stable, and polystable loci}

We now describe the semistable, stable, and polystable loci as unions
of Richardson strata.

\begin{lemma}\label{lem:stability-strata}
Let $\chi$ be a regular dominant character of $T$ lying in the root lattice. Then the semistable, stable and polystable loci of the Peterson variety $\mathrm{Pet}_{n}$ with respect to the linearization $\mathcal{L}(\chi)$ and the one-parameter subgroup $\lambda$ are explicitly given by the following disjoint unions of Richardson strata: 
\[
(\mathrm{Pet}_n)^{ss}_{G_0}(\mathcal L(\chi))
=
\bigsqcup_{\substack{J\subseteq I\subseteq S\\
\langle w_I(\chi),\lambda\rangle\ge0\\
\langle w_J(\chi),\lambda\rangle\le0}}
Y^0_{I,J},
\]
and
\[
(\mathrm{Pet}_n)^s_{G_0}(\mathcal L(\chi))
=
\bigsqcup_{\substack{J\subseteq I\subseteq S\\
\langle w_I(\chi),\lambda\rangle>0\\
\langle w_J(\chi),\lambda\rangle<0}}
Y^0_{I,J}.
\]
Moreover,
\[
(\mathrm{Pet}_n)^{ps}_{G_0}(\mathcal L(\chi))
=
(\mathrm{Pet}_n)^s_{G_0}(\mathcal L(\chi))
\sqcup
\bigsqcup_{\substack{I\subseteq S\\
\langle w_I(\chi),\lambda\rangle=0}}
\{p_I\}.
\]
Equivalently, since \(Y^0_{I,I}=\{p_I\}\),
\[
(\mathrm{Pet}_n)^{ps}_{G_0}(\mathcal L(\chi))
=
(\mathrm{Pet}_n)^s_{G_0}(\mathcal L(\chi))
\sqcup
\bigsqcup_{\substack{I\subseteq S\\
\langle w_I(\chi),\lambda\rangle=0}}
Y^0_{I,I}.
\]
\end{lemma}

\begin{proof}
We know that the Peterson variety admits a decomposition into Richardson cells:
\begin{align}\label{Section 3: Eq 3.1.1}
    \mathrm{Pet}_{n} = \bigsqcup_{J \subseteq I \subseteq S} Y^{0}_{I,J}, 
\end{align}
where $Y^0_{I,J} := B^-w_{J}B/B \cap Bw_{I}B/B \cap \mathrm{Pet}_n$. 

Let \(x\in Y^0_{I,J}\). By
Lemma~\ref{variation of Seshadri's lemma}, the
Hilbert--Mumford numerical functions are
\[
\mu^{\mathcal L(\chi)}(x,\lambda)
=
-\langle w_J(\chi),\lambda\rangle
\]
and
\[
\mu^{\mathcal L(\chi)}(x,-\lambda)
=
\langle w_I(\chi),\lambda\rangle.
\]
Since \(G_0=\lambda(\mathbb G_m)\simeq\mathbb G_m\), by
Corollary~\ref{HM-corollary} it is enough to test the two
one-parameter subgroups \(\lambda\) and \(-\lambda\). Hence \(x\) is
semistable if and only if
\[
\mu^{\mathcal L(\chi)}(x,\lambda)\ge0,
\qquad
\mu^{\mathcal L(\chi)}(x,-\lambda)\ge0,
\]
or equivalently,
\[
\langle w_I(\chi),\lambda\rangle\ge0,
\qquad
\langle w_J(\chi),\lambda\rangle\le0.
\]
Taking the union over all Peterson--Richardson strata gives
\[
(\mathrm{Pet}_n)^{ss}_{G_0}(\mathcal L(\chi))
=
\bigsqcup_{\substack{J\subseteq I\subseteq S\\
\langle w_I(\chi),\lambda\rangle\ge0\\
\langle w_J(\chi),\lambda\rangle\le0}}
Y^0_{I,J}.
\]

Similarly, Corollary~\ref{HM-corollary} shows that \(x\) is stable if
and only if both Hilbert--Mumford inequalities are strict. Thus
\[
\langle w_I(\chi),\lambda\rangle>0,
\qquad
\langle w_J(\chi),\lambda\rangle<0,
\]
and therefore
\[
(\mathrm{Pet}_n)^s_{G_0}(\mathcal L(\chi))
=
\bigsqcup_{\substack{J\subseteq I\subseteq S\\
\langle w_I(\chi),\lambda\rangle>0\\
\langle w_J(\chi),\lambda\rangle<0}}
Y^0_{I,J}.
\]

It remains to determine the polystable locus. A semistable point is
polystable if and only if its \(G_0\)-orbit is closed in the
semistable locus.

First suppose \(I=J\). By
Proposition~\ref{thm:BB-Peterson}, $Y^0_{I,I}=\{p_I\}.$
At \(p_I\), Lemma~\ref{variation of Seshadri's lemma} gives
\[
\mu^{\mathcal L(\chi)}(p_I,\lambda)
=
-\langle w_I(\chi),\lambda\rangle,
\qquad
\mu^{\mathcal L(\chi)}(p_I,-\lambda)
=
\langle w_I(\chi),\lambda\rangle.
\]
Hence, by Corollary~\ref{HM-corollary}, \(p_I\) is semistable if and
only if $\langle w_I(\chi),\lambda\rangle=0.$
Since \(p_I\) is fixed by \(G_0\), its orbit is closed; therefore it is
polystable.

Now suppose \(I\neq J\) and let \(x\in Y^0_{I,J}\) be semistable.
By Proposition~\ref{thm:BB-Peterson},
\[
\lim_{t\to0}\lambda(t)\cdot x=p_J,
\qquad
\lim_{t\to\infty}\lambda(t)\cdot x=p_I.
\]
If $\langle w_J(\chi),\lambda\rangle=0,$
then, by the preceding fixed-point calculation, \(p_J\) is
semistable. Since \(p_J\) lies in the closure of \(G_0\cdot x\) and
\(x\neq p_J\), the orbit \(G_0\cdot x\) is not closed in the
semistable locus. Similarly, if
\[
\langle w_I(\chi),\lambda\rangle=0,
\]
then the semistable fixed point \(p_I\) lies in the closure of
\(G_0\cdot x\), so again the orbit is not closed.

Thus a nonfixed semistable point is polystable precisely when
\[
\langle w_I(\chi),\lambda\rangle>0,
\qquad
\langle w_J(\chi),\lambda\rangle<0,
\]
that is, precisely when it is stable. Consequently,
\[
(\mathrm{Pet}_n)^{ps}_{G_0}(\mathcal L(\chi))
=
(\mathrm{Pet}_n)^s_{G_0}(\mathcal L(\chi))
\sqcup
\bigsqcup_{\substack{I\subseteq S\\
\langle w_I(\chi),\lambda\rangle=0}}
Y^0_{I,I}.
\]
\end{proof}

We record two immediate consequences of the preceding description.
The first shows that the quotient always has the expected dimension,
while the second characterizes those linearizations for which the
GIT quotient is geometric on the whole semistable locus.

\begin{proposition}\label{lem:stability-nonempty}
Let $\chi$ be a regular dominant character of $T$ in the root lattice. Then
\[
(\mathrm{Pet}_n)^s_{G_0}(\mathcal L(\chi))\neq\varnothing.
\]
Consequently,
\[
\dim\left(
(\mathrm{Pet}_n)^{ss}_{G_0}(\mathcal L(\chi))
\sslash G_0
\right)
=
n-2.
\]
\end{proposition}

\begin{proof}
The subsets
\[
\mathrm{Pet}_n\cap Bw_0B/B
\qquad\text{and}\qquad
\mathrm{Pet}_n\cap B^-B/B
\]
are nonempty open subsets of \(\mathrm{Pet}_n\). Since
\(\mathrm{Pet}_n\) is irreducible, their intersection is nonempty.

Let $x\in
\mathrm{Pet}_n\cap Bw_0B/B\cap B^-B/B.$ Equivalently, $x\in Y^0_{S,\emptyset}.$

By the evaluation of the Hilbert-Mumford numerical functions on the Richardson strata, we have:
\[
\mu^{\mathcal L(\chi)}(x,\lambda)
=
-\langle\chi,\lambda\rangle,
\qquad
\mu^{\mathcal L(\chi)}(x,-\lambda)
=
\langle w_0(\chi),\lambda\rangle.
\]
Since \(\chi\) is regular dominant and \(\lambda\) is strictly
antidominant, we have $\langle\chi,\lambda\rangle<0.$
Moreover, by the Weyl group invariance of the pairing,
\[
\langle w_0(\chi),\lambda\rangle
=
\langle\chi,w_0^{-1}(\lambda)\rangle
=
\langle\chi,w_0(\lambda)\rangle>0,
\]
because \(w_0(\lambda)\) is strictly dominant. Hence
\[
\mu^{\mathcal L(\chi)}(x,\lambda)>0,
\qquad
\mu^{\mathcal L(\chi)}(x,-\lambda)>0.
\]
Since \(G_0\simeq\mathbb G_m\), these two inequalities give the
Hilbert--Mumford criterion for stability. Thus $x\in(\mathrm{Pet}_n)^s_{G_0}(\mathcal L(\chi)),$
and the stable locus is nonempty. The stable locus is open in the irreducible variety
\(\mathrm{Pet}_n\), and therefore
\[
\dim(\mathrm{Pet}_n)^s_{G_0}(\mathcal L(\chi))
=
\dim\mathrm{Pet}_n
=
n-1.
\]
Every stable point has finite stabilizer, so every \(G_0\)-orbit in
the stable locus has dimension \(1\). The geometric quotient $(\mathrm{Pet}_n)^s_{G_0}(\mathcal L(\chi))/G_0$ is a nonempty open subset of the GIT quotient $(\mathrm{Pet}_n)^{ss}_{G_0}(\mathcal L(\chi))\sslash G_0.$
Consequently,
\[
\dim\left(
(\mathrm{Pet}_n)^{ss}_{G_0}(\mathcal L(\chi))
\sslash G_0
\right)
=
(n-1)-1
=
n-2.
\]
\end{proof}

\begin{proposition}\label{lem:ss-equals-s}
Let \(\chi\) be a regular dominant root-lattice character. Then
\[
(\mathrm{Pet}_n)^{ss}_{G_0}(\mathcal L(\chi))
=
(\mathrm{Pet}_n)^s_{G_0}(\mathcal L(\chi))
\]
if and only if
\[
\langle w_I(\chi),\lambda\rangle\neq0
\qquad
\text{for every }I\subseteq S.
\]
\end{proposition}

\begin{proof}
Suppose first that $\langle w_I(\chi),\lambda\rangle=0$
for some \(I\subseteq S\). The fixed point $p_I=w_IB/B$
is the Peterson--Richardson stratum $Y^0_{I,I}=\{p_I\}.$
At this point,
\[
\mu^{\mathcal L(\chi)}(p_I,\lambda)
=
-\langle w_I(\chi),\lambda\rangle=0,
\]
and
\[
\mu^{\mathcal L(\chi)}(p_I,-\lambda)
=
\langle w_I(\chi),\lambda\rangle=0.
\]
Hence \(p_I\) is semistable but not stable. Therefore
\[
(\mathrm{Pet}_n)^{ss}_{G_0}(\mathcal L(\chi))
\neq
(\mathrm{Pet}_n)^s_{G_0}(\mathcal L(\chi)).
\]

Conversely, suppose that $\langle w_I(\chi),\lambda\rangle\neq0$ for every $I\subseteq S$, and let $x\in(\mathrm{Pet}_n)^{ss}_{G_0}(\mathcal L(\chi)).$
Then \(x\in Y^0_{I,J}\) for some \(J\subseteq I\subseteq S\), and
semistability gives
\[
\langle w_I(\chi),\lambda\rangle\ge0,
\qquad
\langle w_J(\chi),\lambda\rangle\le0.
\]
By hypothesis neither pairing is zero. Thus
\[
\langle w_I(\chi),\lambda\rangle>0,
\qquad
\langle w_J(\chi),\lambda\rangle<0.
\]
Equivalently,
\[
\mu^{\mathcal L(\chi)}(x,-\lambda)>0,
\qquad
\mu^{\mathcal L(\chi)}(x,\lambda)>0.
\]
Hence \(x\) is stable. Therefore $(\mathrm{Pet}_n)^{ss}_{G_0}(\mathcal L(\chi))
=
(\mathrm{Pet}_n)^s_{G_0}(\mathcal L(\chi)).$
\end{proof}

\begin{remark}\label{rem:geometric-quotient}
Under the equivalent conditions of
Proposition~\ref{lem:ss-equals-s}, every semistable point is stable.
Consequently, every semistable orbit is closed in the semistable
locus and has finite stabilizer, and hence
\[
(\mathrm{Pet}_n)^{ss}_{G_0}(\mathcal L(\chi))
\sslash G_0
\]
is the geometric quotient of the semistable locus.

Let
\[
\pi:
(\mathrm{Pet}_n)^{ss}_{G_0}(\mathcal L(\chi))
\longrightarrow
(\mathrm{Pet}_n)^{ss}_{G_0}(\mathcal L(\chi))
\sslash G_0
\]
be the quotient map. If \(x\) is a smooth stable point, then Luna's
slice theorem describes the quotient \'etale-locally near
\(\pi(x)\) as the quotient of a smooth slice by the finite stabilizer
\((G_0)_x\). Thus any singularity of the quotient arising over the
smooth locus is a finite quotient singularity. Since
\(G_0\simeq\mathbb G_m\), every finite stabilizer is a cyclic group
\(\mu_m\) for some \(m\ge1\). In particular, if \(x\) is smooth and has trivial stabilizer, then
the quotient is smooth at \(\pi(x)\). 
\end{remark}

\subsection{Maximal-parabolic inequalities and the deep chamber}

We now isolate the region of the dominant cone for which every
semistable point lies in the Peterson big cell. By the monotonicity
of the quantities
\(\langle w_I(\chi),\lambda\rangle\), it is enough to analyze the
maximal proper subsets. 
We begin with an explicit computation of the corresponding linear
forms. Let
\[
\mathcal D
=
\{\chi\in Q_{\mathbb R}:\chi\text{ is regular dominant}\},
\]
and for \(1\le r\le n-1\) set
\[
K_r:=S\setminus\{\alpha_r\}.
\]

\begin{lemma}\label{lem:maximal-parabolic-weights}
For \(1\le r,j\le n-1\),
\[
\bigl\langle w_{K_r}(\varpi_j),\lambda\bigr\rangle
=
\begin{cases}
\dfrac{j(2r-n-j)}2, & j\le r,\\[6pt]
\dfrac{(n-j)(j-2r)}2, & j>r.
\end{cases}
\]
Equivalently,
\[
\bigl\langle w_{K_r}(\varpi_j),\lambda\bigr\rangle
=
\begin{cases}
\dfrac{j(n-j)}2-j(n-r), & j\le r,\\[6pt]
\dfrac{j(n-j)}2-r(n-j), & j>r.
\end{cases}
\]
\end{lemma}

\begin{proof}
Recall that
\[
\varpi_j
=
\epsilon_1+\cdots+\epsilon_j
-\frac{j}{n}\sum_{i=1}^n\epsilon_i,
\qquad
\langle\epsilon_i,\lambda\rangle=i,
\]
and
\[
w_{K_r}
=
(r,r-1,\ldots,1,n,n-1,\ldots,r+1).
\]

If \(j\le r\), then
\[
w_{K_r}(\epsilon_1+\cdots+\epsilon_j)
=
\epsilon_r+\cdots+\epsilon_{r-j+1},
\]
so
\[
\begin{aligned}
\bigl\langle w_{K_r}(\varpi_j),\lambda\bigr\rangle
&=
\sum_{\ell=0}^{j-1}(r-\ell)
-\frac{j(n+1)}2  \\
&=
\frac{j(2r-n-j)}2.
\end{aligned}
\]

If \(j>r\), then
\[
w_{K_r}(\epsilon_1+\cdots+\epsilon_j)
=
\epsilon_1+\cdots+\epsilon_r
+\epsilon_n+\cdots+\epsilon_{n-j+r+1},
\]
and hence
\[
\begin{aligned}
\bigl\langle w_{K_r}(\varpi_j),\lambda\bigr\rangle
&=
\frac{r(r+1)}2
+n(j-r)
-\frac{(j-r-1)(j-r)}2
-\frac{j(n+1)}2 \\
&=
\frac{(n-j)(j-2r)}2.
\end{aligned}
\]
The equivalent forms follow by elementary simplification.
\end{proof}

\begin{lemma}\label{lem:maximal-parabolic-walls}
For \(1\le r\le n-1\),
\[
\langle w_{K_r}(\chi),\lambda\rangle<0
\qquad
\text{for every }\chi\in\mathcal D
\]
if and only if
\[
n-1\le2r\le n+1.
\]
Equivalently, the hyperplane
\[
H_{K_r}
=
\{\chi\in Q_{\mathbb R}:
\langle w_{K_r}(\chi),\lambda\rangle=0\}
\]
meets \(\mathcal D\) if and only if
\[
2r<n-1
\qquad\text{or}\qquad
2r>n+1.
\]
\end{lemma}

\begin{proof}
Write
\[
\chi=\sum_{j=1}^{n-1}m_j\varpi_j,
\qquad m_j>0.
\]
By Lemma~\ref{lem:maximal-parabolic-weights},
\[
\langle w_{K_r}(\chi),\lambda\rangle
=
\sum_{j=1}^{n-1}m_jc_j(r),
\]
where
\[
c_j(r)=
\begin{cases}
\dfrac{j(2r-n-j)}2, & j\le r,\\[6pt]
\dfrac{(n-j)(j-2r)}2, & j>r.
\end{cases}
\]

For \(j\le r\), all \(c_j(r)\le0\) if and only if $2r\le n+1,$ while for \(j>r\), all \(c_j(r)\le0\) if and only if $n-1\le2r.$ 
Thus all coefficients are non-positive precisely when
\[
n-1\le2r\le n+1.
\]
They cannot all vanish, since
\[
c_1(r)=r-\frac{n+1}{2},
\qquad
c_{n-1}(r)=\frac{n-1}{2}-r
\]
whenever \(r<n-1\), and these have sum \(-1\). Hence $\langle w_{K_r}(\chi),\lambda\rangle<0$ for every \(m_j>0\).

Conversely, if \(2r<n-1\), then
\[
c_1(r)<0<c_{n-1}(r),
\]
whereas if \(2r>n+1\), then
\[
c_{n-1}(r)<0<c_1(r).
\]
Thus the linear form
\[
\chi\longmapsto
\langle w_{K_r}(\chi),\lambda\rangle
\]
takes both signs on the convex cone \(\mathcal D\), so its zero
hyperplane meets \(\mathcal D\).
\end{proof}

The preceding computation allows us to characterize the region in
which all proper Peterson cells are unstable.

\begin{proposition}\label{prop:deep-chamber-characterization}
Let \(\chi\) be a regular dominant character of \(T\) lying in the
root lattice. Then the following are equivalent:
\begin{enumerate}
\item
\[
(\mathrm{Pet}_n)^{ss}_{G_0}(\mathcal L(\chi))
\subseteq
\mathrm{Pet}_n\cap Bw_0B/B;
\]

\item
\[
\langle w_I(\chi),\lambda\rangle<0
\qquad\text{for every }I\subsetneq S;
\]

\item
\[
\langle w_{K_r}(\chi),\lambda\rangle<0
\qquad(1\le r\le n-1).
\]
\end{enumerate}
\end{proposition}

\begin{proof}
We know that the Peterson variety is stratified by Richardson strata indexed by $J \subseteq I \subseteq S$:
\[ \mathrm{Pet}_{n} = \bigsqcup_{J \subseteq I \subseteq S} Y^{0}_{I,J} \]
An element $x \in Y^0_{I,J}$ belongs to $Bw_{0}B/B$ if and only if $I = S$, since $w_{S} = w_0$ is the longest element of the Weyl group. 
 We first show that \(Y^0_{I,\emptyset}\) is nonempty for every \(I\subseteq S\).
Recall that
\[
Y^0_{I,\emptyset}
=
X_I^0\cap B^-B/B.
\]
By the Peterson-cell closure relation,
\[
X_I=\overline{X_I^0}
=
\bigsqcup_{J\subseteq I}X_J^0.
\]
Since \(\emptyset\subseteq I\), the fixed point $p_\emptyset=eB/B$
belongs to \(X_I\). On the other hand, \(B^-B/B\) is an open
neighbourhood of \(p_\emptyset\) in \(G/B\). Therefore $X_I\cap B^-B/B$ is a nonempty open subset of the irreducible variety \(X_I\).
Since \(X_I^0\) is dense in \(X_I\), this open subset meets \(X_I^0\).
Hence
\[
Y^0_{I,\emptyset}
=
X_I^0\cap B^-B/B
\neq\emptyset.
\]
Assume first that
\[
(\mathrm{Pet}_n)^{ss}_{G_0}(\mathcal L(\chi))
\subseteq
\mathrm{Pet}_n\cap Bw_0B/B.
\]
Suppose, to the contrary, that there exists a proper subset
\(I\subsetneq S\) such that $\langle w_I(\chi),\lambda\rangle\ge0.$ Since $Y^0_{I,\emptyset}\neq\emptyset$, for every \(x\in Y^0_{I,\emptyset}\), we have 
\[
\mu^{\mathcal L(\chi)}(x,\lambda)
=
-\langle\chi,\lambda\rangle
\]
and
\[
\mu^{\mathcal L(\chi)}(x,-\lambda)
=
\langle w_I(\chi),\lambda\rangle.
\]
Since \(\chi\) is regular dominant and \(\lambda\) is strictly
antidominant, we have $\langle\chi,\lambda\rangle<0.$
Thus
\[
\mu^{\mathcal L(\chi)}(x,\lambda)>0,
\qquad
\mu^{\mathcal L(\chi)}(x,-\lambda)\ge0,
\]
so every point of \(Y^0_{I,\emptyset}\) is semistable.

However, $Y^0_{I,\emptyset}\subseteq Bw_IB/B,$
and since \(I\neq S\), the Schubert cells \(Bw_IB/B\) and
\(Bw_0B/B\) are disjoint. Hence \(Y^0_{I,\emptyset}\) contains a
semistable point outside \(Bw_0B/B\), contradicting (1). Therefore
\[
\langle w_I(\chi),\lambda\rangle<0
\qquad
\text{for every }I\subsetneq S.
\]
Thus \((1)\Rightarrow(2)\).

Conversely, suppose that (2) holds and let $x\in(\mathrm{Pet}_n)^{ss}_{G_0}(\mathcal L(\chi)).$
So \(x\) belongs to a unique
stratum
\[
Y^0_{I,J},
\qquad
J\subseteq I\subseteq S.
\]
Then semistability implies that 
\[
\langle w_I(\chi),\lambda\rangle\ge0.
\]
If \(I\subsetneq S\), this contradicts (2). Hence necessarily $I=S.$ Therefore
\[
x\in Y^0_{S,J}\subseteq Bw_0B/B,
\]
and so
\[
(\mathrm{Pet}_n)^{ss}_{G_0}(\mathcal L(\chi))
\subseteq
\mathrm{Pet}_n\cap Bw_0B/B.
\]
Thus \((2)\Rightarrow(1)\).

Since $K_r=S\setminus\{\alpha_r\}$ is a proper subset of \(S\) for every \(1\le r\le n-1\), the
implication
\[
(2)\Rightarrow(3)
\]
is immediate.

Finally, suppose that (3) holds and let \(I\subsetneq S\). Since
\(I\) is proper, there exists some \(r\in\{1,\ldots,n-1\}\) such that $\alpha_r\notin I$. Hence $I\subseteq K_r.$ Then $w_I\le w_{K_r}$ in the Bruhat order and since \(\chi\) is dominant, we have  
\[
\langle w_I(\chi),\lambda\rangle
\le
\langle w_{K_r}(\chi),\lambda\rangle
<0.
\]
Thus $\langle w_I(\chi),\lambda\rangle<0$ for every $I\subsetneq S,$
which proves \((3)\Rightarrow(2)\).
\end{proof}

\begin{definition}\label{def:deep-chamber}
We define the \emph{deep GIT chamber} to be 
\[
\mathcal C_{\mathrm{deep}}^\lambda
:=
\left\{
\chi\in\mathcal D:
\langle w_{K_r}(\chi),\lambda\rangle<0
\text{ for every }1\le r\le n-1
\right\}.
\]
By Proposition~\ref{prop:deep-chamber-characterization},
\[
\mathcal C_{\mathrm{deep}}^\lambda
=
\left\{
\chi\in\mathcal D:
\langle w_I(\chi),\lambda\rangle<0
\text{ for every }I\subsetneq S
\right\}.
\]
Thus, for a regular dominant root-lattice character \(\chi\),
\[
\chi\in\mathcal C_{\mathrm{deep}}^\lambda
\quad\Longleftrightarrow\quad
(\mathrm{Pet}_n)^{ss}_{G_0}(\mathcal L(\chi))
\subseteq
\mathrm{Pet}_n\cap Bw_0B/B.
\]
\end{definition}

If $\chi=\sum_{j=1}^{n-1}c_j\varpi_j,
 \,\, c_j>0,$
then the deep chamber is equivalently determined by the \(n-1\)
linear inequalities
\[
\sum_{j=1}^{n-1}
c_j\,
\langle w_{K_r}(\varpi_j),\lambda\rangle<0,
\qquad
1\le r\le n-1,
\]
whose coefficients are given explicitly in
Lemma~\ref{lem:maximal-parabolic-weights}.

\begin{remark}\label{rem:deep-nonempty}
The deep chamber is nonempty for every \(n\). Indeed, let
\[
\alpha_0
=
\epsilon_1-\epsilon_n
=
\varpi_1+\varpi_{n-1}
\]
be the highest root. For \(K_r=S\setminus\{\alpha_r\}\), one has $w_{K_r}(\alpha_0)=\alpha_r,$
and hence
\[
\langle w_{K_r}(\alpha_0),\lambda\rangle
=
\langle\alpha_r,\lambda\rangle
=
-1.
\]
Therefore, for every regular dominant root-lattice character \(\chi\),
\[
\langle
w_{K_r}(\chi+m\alpha_0),\lambda
\rangle
=
\langle w_{K_r}(\chi),\lambda\rangle-m<0
\]
for all \(1\le r\le n-1\) and all sufficiently large \(m\). Thus $\chi+m\alpha_0\in\mathcal C_{\mathrm{deep}}^\lambda$
for all sufficiently large \(m\), and consequently $\mathcal C_{\mathrm{deep}}^\lambda\neq\varnothing.$

For \(n=3\), the situation is stronger. If $\chi=m_1\varpi_1+m_2\varpi_2,
\qquad m_1,m_2>0,$
then
\[
\langle w_{K_1}(\chi),\lambda\rangle=-m_1,
\qquad
\langle w_{K_2}(\chi),\lambda\rangle=-m_2.
\]
Hence every regular dominant character lies in the deep chamber, so in this case we have 
\[
\mathcal C_{\mathrm{deep}}^\lambda=\mathcal D.
\]
\end{remark}

\begin{remark}\label{rem:deep-is-chamber}
The terminology ``deep GIT chamber'' is justified. Indeed,
\(\mathcal C_{\mathrm{deep}}^\lambda\) is an open convex subset of
\(\mathcal D\), being defined by finitely many strict linear
inequalities. By
Proposition~\ref{prop:deep-chamber-characterization}, for every
\(\chi\in\mathcal C_{\mathrm{deep}}^\lambda\), $\langle w_I(\chi),\lambda\rangle<0, (I\subsetneq S),$ while $\langle w_0(\chi),\lambda\rangle>0.$
Hence no GIT wall
\[
H_I
=
\{\chi\in\mathcal D:
\langle w_I(\chi),\lambda\rangle=0\}
\]
meets \(\mathcal C_{\mathrm{deep}}^\lambda\).

Thus \(\mathcal C_{\mathrm{deep}}^\lambda\) is contained in a single
GIT chamber. Conversely, throughout that chamber all the linear forms
\(\langle w_I(\,\cdot\,),\lambda\rangle\), \(I\subsetneq S\), have
the same negative sign. Proposition~\ref{prop:deep-chamber-characterization}
therefore shows that the chamber is contained in
\(\mathcal C_{\mathrm{deep}}^\lambda\). Hence
\(\mathcal C_{\mathrm{deep}}^\lambda\) is itself a GIT chamber.
\end{remark}

\subsection{The deep-chamber quotient}
We now determine the GIT quotient for a root-lattice character in the
deep chamber. The preceding characterization shows first that the
semistable locus is concentrated in the Peterson big cell; the
explicit Toeplitz coordinates on that cell then identify the quotient
with a weighted projective space.

\begin{proposition}\label{prop:deep-semistable-locus}
Let \(\chi\) be a regular dominant root-lattice character such that
\[
\chi\in\mathcal C_{\mathrm{deep}}^\lambda.
\]
Then
\[
(\mathrm{Pet}_n)^{ss}_{G_0}(\mathcal L(\chi))
=
(\mathrm{Pet}_n)^s_{G_0}(\mathcal L(\chi))
=
\bigl(\mathrm{Pet}_n\cap Bw_0B/B\bigr)
\setminus\{w_0B/B\}.
\]
\end{proposition}

\begin{proof}
We have the stratification of the Peterson variety into Richardson strata:
\[ \mathrm{Pet}_{n} = \bigsqcup_{J \subseteq I \subseteq S} Y^{0}_{I,J}. \]
A stratum $Y^{0}_{I,J}$ is semistable if and only if $\langle w_{I}(\chi), \lambda \rangle \geq 0$ and $\langle w_{J}(\chi), \lambda \rangle \leq 0$. By Proposition~\ref{prop:deep-chamber-characterization},
\[
\langle w_I(\chi),\lambda\rangle<0
\qquad
\text{for every }I\subsetneq S.
\]
Hence no Peterson--Richardson stratum \(Y^0_{I,J}\) with
\(I\subsetneq S\) can contain a semistable point. 
Therefore
\[
(\mathrm{Pet}_n)^{ss}_{G_0}(\mathcal L(\chi))
\subseteq
\bigsqcup_{J\subseteq S}Y^0_{S,J}
=
\mathrm{Pet}_n\cap Bw_0B/B.
\]

Now let \(x\in Y^0_{S,J}\). The Hilbert--Mumford weights are
\[
\mu^{\mathcal L(\chi)}(x,-\lambda)
=
\langle w_0(\chi),\lambda\rangle
\]
and
\[
\mu^{\mathcal L(\chi)}(x,\lambda)
=
-\langle w_J(\chi),\lambda\rangle.
\]
Since \(\chi\) is regular dominant and \(\lambda\) is strictly
antidominant, $\langle w_0(\chi),\lambda\rangle>0.$ If \(J\subsetneq S\), then, since
\(\chi\in\mathcal C_{\mathrm{deep}}^\lambda\),
\[
\langle w_J(\chi),\lambda\rangle<0.
\]
Thus both Hilbert--Mumford weights are strictly positive, and every
point of \(Y^0_{S,J}\) is stable.

On the other hand, $Y^0_{S,S}=\{w_0B/B\},$ and at this fixed point
\[
\mu^{\mathcal L(\chi)}(w_0B/B,\lambda)
=
-\langle w_0(\chi),\lambda\rangle<0.
\]
So \(w_0B/B\) is unstable. Combining these cases, the semistable locus is exactly the union of $Y^0_{S,J}$ for all $J \subsetneq S$. This union is precisely the intersection of $\mathrm{Pet}_n$ with the big cell, excluding the single point $w_0B/B$. Hence,
\[
(\mathrm{Pet}_n)^{ss}_{G_0}(\mathcal L(\chi))
=
(\mathrm{Pet}_n)^s_{G_0}(\mathcal L(\chi))
=
\bigl(\mathrm{Pet}_n\cap Bw_0B/B\bigr)
\setminus\{w_0B/B\}.
\]
\end{proof}

\begin{theorem}\label{thm:deep-quotient}
Let \(\chi\) be a regular dominant root-lattice character such that $\chi\in\mathcal C_{\mathrm{deep}}^\lambda.$
Then
\[
(\mathrm{Pet}_n)^{ss}_{G_0}(\mathcal L(\chi))
\sslash G_0
\cong
\mathbb P(1,2,\ldots,n-1).
\]
\end{theorem}

\begin{proof}
For $\mathbf a=(a_1,\ldots,a_{n-1})\in\mathbb C^{n-1},$
we set
\[
u(\mathbf a)
=
I+
\sum_{i=1}^{n-1}
a_i
\left(
\sum_{j=1}^{n-i}E_{j,j+i}
\right),
\]
where \(E_{j,k}\) denotes the elementary matrix with a \(1\) in the
\((j,k)\)-entry and zeros elsewhere. Thus \(u(\mathbf a)\) is the upper triangular unipotent Toeplitz
matrix whose \(i\)-th superdiagonal is constant with value \(a_i\).

By \cite[Lemma~9(IV)]{Insko},
\[
\mathrm{Pet}_n\cap Bw_0B/B
=
\left\{
u(\mathbf a)w_0B/B
\;\middle|\;
\mathbf a\in\mathbb C^{n-1}
\right\}.
\]
Hence Proposition~\ref{prop:deep-semistable-locus} gives
\[
(\mathrm{Pet}_n)^{ss}_{G_0}(\mathcal L(\chi))
=
\left\{
u(\mathbf a)w_0B/B
\;\middle|\;
\mathbf a\in\mathbb C^{n-1}\setminus\{\mathbf0\}
\right\},
\]
since \(w_0B/B\) corresponds to \(\mathbf a=\mathbf0\).

Thus
\[
\phi:
(\mathrm{Pet}_n)^{ss}_{G_0}(\mathcal L(\chi))
\longrightarrow
\mathbb C^{n-1}\setminus\{\mathbf0\},
\qquad
u(\mathbf a)w_0B/B\longmapsto\mathbf a,
\]
is an isomorphism.

We next compute the \(G_0\)-action in these coordinates. Since $\lambda(t)
=
\operatorname{diag}(t,t^2,\ldots,t^n),$ we have
\[
\lambda(t)E_{j,j+i}\lambda(t)^{-1}
=
t^{-i}E_{j,j+i}.
\]
Moreover,
\[
\lambda(t)w_0B/B=w_0B/B,
\]
because \(w_0^{-1}\lambda(t)w_0\in T\subseteq B\). Therefore
\[
\lambda(t)\cdot
u(a_1,\ldots,a_{n-1})w_0B/B
=
u(t^{-1}a_1,t^{-2}a_2,\ldots,t^{-(n-1)}a_{n-1})
w_0B/B.
\]
Consequently, under \(\phi\), the \(G_0\)-action becomes
\[
t\cdot(a_1,\ldots,a_{n-1})
=
(t^{-1}a_1,t^{-2}a_2,\ldots,t^{-(n-1)}a_{n-1}).
\]
Composing the parameterization of \(G_0\simeq\mathbb G_m\) with the
automorphism \(t\mapsto t^{-1}\), this is the standard weighted action
with weights
\[
1,2,\ldots,n-1.
\]

By Proposition~\ref{prop:deep-semistable-locus}, stable and semistable
loci coincide. Hence the GIT quotient is the geometric quotient of
\(\mathbb C^{n-1}\setminus\{\mathbf0\}\) by this weighted
\(\mathbb G_m\)-action. Therefore
\[
(\mathrm{Pet}_n)^{ss}_{G_0}(\mathcal L(\chi))
\sslash G_0
\cong
\bigl(\mathbb C^{n-1}\setminus\{\mathbf0\}\bigr)
\big/\mathbb G_m
\cong
\mathbb P(1,2,\ldots,n-1).
\]
\end{proof}

\begin{example}\label{ex:GL3-deep}
For \(G=\mathrm{GL}(3,\mathbb C)\) we know that $\mathcal C_{\mathrm{deep}}^\lambda=\mathcal D.$ So for every regular dominant root-lattice character
\(\chi\),
\[
(\mathrm{Pet}_3)^{ss}_{G_0}(\mathcal L(\chi))
=
(\mathrm{Pet}_3)^s_{G_0}(\mathcal L(\chi)),
\]
and Theorem~\ref{thm:deep-quotient} gives
\[
(\mathrm{Pet}_3)^{ss}_{G_0}(\mathcal L(\chi))
\sslash G_0
\cong
\mathbb P(1,2)
\cong
\mathbb P^1.
\]
Hence in this case the GIT quotient is independent of the choice of regular
dominant root-lattice character \(\chi\).
\end{example}

\begin{example}\label{exam:not-in-deep}
Let \(G=\mathrm{GL}(4,\mathbb C)\). We show that the deep GIT chamber
\(\mathcal C_{\mathrm{deep}}^\lambda\) is a proper subset of the
regular dominant cone.

Consider the character
\[
\chi
=
4\epsilon_1+3\epsilon_2+2\epsilon_3-9\epsilon_4.
\]
Then \(\chi\) lies in the root lattice and is regular
dominant. We take $K_1=S\setminus\{\alpha_1\}
=
\{\alpha_2,\alpha_3\}.$
The longest element \(w_{K_1}\) reverses the indices
\(2,3,4\), and therefore
\[
w_{K_1}(\chi)
=
4\epsilon_1-9\epsilon_2+2\epsilon_3+3\epsilon_4.
\]
Since $
\lambda(t)=\operatorname{diag}(t,t^2,t^3,t^4),$
we obtain
\[
\begin{aligned}
\langle w_{K_1}(\chi),\lambda\rangle
&=
4-18+6+12 >0 \end{aligned}
\]
Hence the defining inequality $\langle w_{K_1}(\chi),\lambda\rangle<0$
for the deep chamber fails. Thus $\chi\notin\mathcal C_{\mathrm{deep}}^\lambda.$

In particular, in this case 
\[
\mathcal C_{\mathrm{deep}}^\lambda
\subsetneq\mathcal D.
\]
As we shall see in Example~\ref{examp:not-weighted}, for this
linearization additional boundary strata become semistable, and the
resulting GIT quotient is not isomorphic to
\(\mathbb P(1,2,3)\).
\end{example}

\subsection{Variation of GIT and wall crossing} The preceding description of the semistable locus shows that its variation with the choice of linearization is controlled by the finitely many linear forms
\[
\chi\longmapsto \langle w_I(\chi),\lambda\rangle,
\qquad I\subseteq S.
\]
We therefore obtain the following chamber decomposition and wall-crossing description for the Peterson GIT quotients.

For \(I\subseteq S\), let
\[
H_I
:=
\left\{
\chi\in\mathcal D:
\langle w_I(\chi),\lambda\rangle=0
\right\}.
\]
The connected components of
\[
\mathcal D\setminus\bigcup_{I\subseteq S}H_I
\]
will be called the \emph{GIT chambers}. For a regular dominant
root-lattice character \(\chi\), set
\[
X_\chi
:=
(\mathrm{Pet}_n)^{ss}_{G_0}(\mathcal L(\chi))
\sslash G_0.
\]

\begin{theorem}\label{thm:peterson-vgit}
The Peterson GIT quotients satisfy the following variation-of-GIT
properties.

\begin{enumerate}
\item
If two regular dominant root-lattice characters
\(\chi,\chi'\) lie in the same GIT chamber, then
\[
(\mathrm{Pet}_n)^{ss}_{G_0}(\mathcal L(\chi))
=
(\mathrm{Pet}_n)^{ss}_{G_0}(\mathcal L(\chi')),
\]
and there is a canonical isomorphism
\[
X_\chi\cong X_{\chi'}.
\]

\item
Let \(C_+\) and \(C_-\) be adjacent GIT chambers with common wall
\(F\). Choose regular dominant root-lattice characters
\[
\chi_+\in C_+,\qquad
\chi_-\in C_-,
\]
and a root-lattice character \(\chi_0\) in the relative interior of
\(F\). Then there are projective birational morphisms
\[
X_{\chi_+}\longrightarrow X_{\chi_0}
\longleftarrow X_{\chi_-}.
\]
In particular, \(X_{\chi_+}\) and \(X_{\chi_-}\) are related by a
birational VGIT wall crossing
\[
X_{\chi_+}\dashrightarrow X_{\chi_-}.
\]

\item
If \(\chi\) lies on no wall, equivalently
\[
\langle w_I(\chi),\lambda\rangle\neq0
\qquad
\text{for every }I\subseteq S,
\]
then
\[
(\mathrm{Pet}_n)^{ss}_{G_0}(\mathcal L(\chi))
=
(\mathrm{Pet}_n)^s_{G_0}(\mathcal L(\chi)),
\]
and \(X_\chi\) is the geometric quotient of the semistable locus.
\end{enumerate}
\end{theorem}

\begin{proof}
For a Richardson stratum \(Y^0_{I,J}\), the
Hilbert-Mumford criterion gives
\[
Y^0_{I,J}
\subseteq
(\mathrm{Pet}_n)^{ss}_{G_0}(\mathcal L(\chi))
\]
if and only if
\[
\langle w_I(\chi),\lambda\rangle\ge0,
\qquad
\langle w_J(\chi),\lambda\rangle\le0.
\]
Thus the semistable locus depends only on the signs of the finitely
many linear forms
\[
\langle w_I(\,\cdot\,),\lambda\rangle.
\]
Their zero sets therefore cut \(\mathcal D\) into finitely many open
convex polyhedral chambers.

If \(\chi\) and \(\chi'\) lie in the same chamber, all these linear
forms have the same signs at \(\chi\) and \(\chi'\). Hence exactly the
same Richardson strata are semistable, and so
\[
(\mathrm{Pet}_n)^{ss}_{G_0}(\mathcal L(\chi))
=
(\mathrm{Pet}_n)^{ss}_{G_0}(\mathcal L(\chi')).
\]
Both \(X_\chi\) and \(X_{\chi'}\) are good quotients of this same
\(G_0\)-variety. By uniqueness of good quotients, they are canonically
isomorphic. This proves (1).

Now let \(C_+\) and \(C_-\) be adjacent chambers with common wall
\(F\), and choose \(\chi_+,\chi_-,\chi_0\) as in (2). At \(\chi_0\),
the inequalities corresponding to the wall become weak equalities,
while all other signs agree with those in the adjacent chambers.
Consequently,
\[
(\mathrm{Pet}_n)^{ss}_{G_0}(\mathcal L(\chi_\pm))
\subseteq
(\mathrm{Pet}_n)^{ss}_{G_0}(\mathcal L(\chi_0)).
\]
The standard variation-of-GIT construction
\cite{DH,Thaddeus} therefore gives projective morphisms
\[
X_\pm\longrightarrow X_0.
\]

These morphisms are birational. Indeed, the stratum
\[
Y^0_{S,\emptyset}
=
\mathrm{Pet}_n\cap Bw_0B/B\cap B^-B/B
\]
is a nonempty dense open \(G_0\)-invariant subset of
\(\mathrm{Pet}_n\), and for every regular dominant \(\chi\),
\[
\langle\chi,\lambda\rangle<0,
\qquad
\langle w_0(\chi),\lambda\rangle>0.
\]
Hence every point of \(Y^0_{S,\emptyset}\) is stable for every
linearization under consideration. Its geometric quotient is therefore
a common dense open subset of \(X_+\), \(X_-\), and \(X_0\).
Thus
\[
X_+\longrightarrow X_0
\longleftarrow X_-
\]
are birational, proving (2).

Finally, suppose that \(\chi\) avoids all walls. If
\(x\in Y^0_{I,J}\) is semistable, then
\[
\langle w_I(\chi),\lambda\rangle\ge0,
\qquad
\langle w_J(\chi),\lambda\rangle\le0.
\]
Since neither pairing can vanish, both inequalities are strict.
Hence every semistable point is stable. Therefore
\[
(\mathrm{Pet}_n)^{ss}_{G_0}(\mathcal L(\chi))
=
(\mathrm{Pet}_n)^s_{G_0}(\mathcal L(\chi)),
\]
and the quotient is geometric. This proves (3).
\end{proof}

\begin{remark}\label{rem:quotients-birational}
For every regular dominant root-lattice character \(\chi\), the
Peterson--Richardson stratum
\[
Y^0_{S,\emptyset}
=
\mathrm{Pet}_n\cap Bw_0B/B\cap B^-B/B
\]
is contained in $(\mathrm{Pet}_n)^s_{G_0}(\mathcal L(\chi)).$
Indeed,
\[
\langle\chi,\lambda\rangle<0,
\qquad
\langle w_0(\chi),\lambda\rangle>0.
\]
Moreover, \(Y^0_{S,\emptyset}\) is a dense open \(G_0\)-invariant
subset of \(\mathrm{Pet}_n\). Hence its geometric quotient $Y^0_{S,\emptyset}/G_0$ is a common dense open subset of every Peterson GIT quotient
\(X_\chi\). Consequently, all the varieties \(X_\chi\) are
birational.

In particular, by Theorem~\ref{thm:deep-quotient}, every \(X_\chi\)
is birational to $\mathbb P(1,2,\ldots,n-1).$
Thus every Peterson GIT quotient considered here is rational.
Moreover, by Lemma~\ref{lem:stability-nonempty},
\[
\dim X_\chi=n-2.
\]
\end{remark}

\section{Normality of the Quotient}
\label{sec:normality}

In this section we determine exactly when the Peterson GIT quotient
\[
X_\chi
:=
(\mathrm{Pet}_n)^{ss}_{G_0}(\mathcal L(\chi))
\sslash G_0
\]
is normal. We first determine the normal locus of
\(\mathrm{Pet}_n\). We then analyze the maximal-parabolic wall strata
\(Y^0_{S,K_r}\) by means of the Toeplitz coordinates on the Peterson
big cell. This yields a complete normality criterion, including
linearizations lying on GIT walls.

\subsection{The normal locus of the Peterson variety} We begin with the local geometry of the Peterson variety itself.
The description of its singular locus, together with Serre's
criterion, gives an explicit description of its normal locus.

By \cite[Theorem~4]{Insko}, the singular locus of the Peterson variety is
\[
\operatorname{Sing}(\mathrm{Pet}_n)
=
\bigsqcup_{\substack{I\subseteq S\\
I\notin
\{S,\,
S\setminus\{\alpha_1\},\,
S\setminus\{\alpha_{n-1}\}\}}}
X_I^0.
\]
Thus
\[
\mathrm{Pet}_n^{\mathrm{sm}}
=
X_S^0
\sqcup
X_{S\setminus\{\alpha_1\}}^0
\sqcup
X_{S\setminus\{\alpha_{n-1}\}}^0.
\]

For \(n\le3\), the Peterson variety is normal. For \(n\ge4\), its
normal locus is described as follows.

\begin{proposition}\label{prop:normal-locus}
Assume \(n\ge4\). Then
\[
\mathrm{Pet}_n^{\mathrm{norm}}
=
X_S^0
\sqcup
X_{S\setminus\{\alpha_1\}}^0
\sqcup
X_{S\setminus\{\alpha_{n-1}\}}^0
\sqcup
X_{S\setminus\{\alpha_1,\alpha_{n-1}\}}^0.
\]
Equivalently,
\[
\mathrm{Pet}_n^{\mathrm{norm}}
=
\mathrm{Pet}_n^{\mathrm{sm}}
\sqcup
X_{S\setminus\{\alpha_1,\alpha_{n-1}\}}^0.
\]
\end{proposition}

\begin{proof}
By \cite[Corollary~7]{Insko}, \(\mathrm{Pet}_n\) is a local complete
intersection and hence Cohen--Macaulay. Therefore it satisfies
Serre's condition \(S_2\). Thus, by Serre's criterion, normality is equivalent to
regularity in codimension one.

Recall that $X_I^0\simeq\mathbb A^{|I|}$ and $\dim\mathrm{Pet}_n=n-1.$
Hence
\[
\operatorname{codim}_{\mathrm{Pet}_n}X_I^0
=
n-1-|I|.
\]
In particular, the codimension-one Peterson cells are exactly $X_{S\setminus\{\alpha_i\}}^0, 1\le i\le n-1.$ By the description of the singular locus above, the cells with
\(i=1\) and \(i=n-1\) are smooth, whereas $X_{S\setminus\{\alpha_i\}}^0,$ for $2\le i\le n-2,$ are singular. So the codimension-one Peterson cells contained in
\(\operatorname{Sing}(\mathrm{Pet}_n)\) are precisely $
X_{S\setminus\{\alpha_i\}}^0$ where $2\le i\le n-2.$
Consequently, the codimension-one irreducible components of the
singular locus are $\overline{X_{S\setminus\{\alpha_i\}}^0}$, where $2\le i\le n-2.$ Thus the singular codimension-one locus is $\bigsqcup_{i=2}^{n-2}
X_{S\setminus\{\alpha_i\}}^0.$

Since \(\mathrm{Pet}_n\) satisfies \(S_2\), a point is non-normal
precisely when it lies in the closure of a singular codimension-one
point. Therefore the non-normal locus is $\bigcup_{i=2}^{n-2}
\overline{X_{S\setminus\{\alpha_i\}}^0}.$
Using the Peterson-cell closure relation
\[
\overline{X_I^0}
=
\bigsqcup_{J\subseteq I}X_J^0,
\]
we obtain
\[
\mathrm{Pet}_n\setminus\mathrm{Pet}_n^{\mathrm{norm}}
=
\bigsqcup_{\substack{J\subseteq S\\
J\subseteq S\setminus\{\alpha_i\}
\text{ for some }2\le i\le n-2}}
X_J^0.
\]
Equivalently, \(X_J^0\) is contained in the non-normal locus if and
only if \(J\) omits at least one of the middle simple roots $\alpha_2,\ldots,\alpha_{n-2}.$
Hence \(X_J^0\) lies in the normal locus if and only if
\[
\{\alpha_2,\ldots,\alpha_{n-2}\}\subseteq J.
\]
There are exactly four such subsets of \(S\), according to whether
the two endpoint roots \(\alpha_1\) and \(\alpha_{n-1}\) are included:
\[
S,\qquad
S\setminus\{\alpha_1\},\qquad
S\setminus\{\alpha_{n-1}\},\qquad
S\setminus\{\alpha_1,\alpha_{n-1}\}.
\]
Therefore
\[
\mathrm{Pet}_n^{\mathrm{norm}}
=
X_S^0
\sqcup
X_{S\setminus\{\alpha_1\}}^0
\sqcup
X_{S\setminus\{\alpha_{n-1}\}}^0
\sqcup
X_{S\setminus\{\alpha_1,\alpha_{n-1}\}}^0.
\]
The first three cells form \(\mathrm{Pet}_n^{\mathrm{sm}}\), giving
the equivalent description.
\end{proof}

For \(n\ge4\), set
\[
\mathcal N
=
\left\{
S,\,
S\setminus\{\alpha_1\},\,
S\setminus\{\alpha_{n-1}\},\,
S\setminus\{\alpha_1,\alpha_{n-1}\}
\right\}.
\]
Then
\begin{equation}\label{eq:normal-index-set}
\mathrm{Pet}_n^{\mathrm{norm}}
=
\bigsqcup_{I\in\mathcal N}X_I^0.
\end{equation}
Recalling \(K_r=S\setminus\{\alpha_r\}\), we also have
\begin{equation}\label{eq:bad-index-normality}
I\notin\mathcal N
\quad\Longleftrightarrow\quad
I\subseteq K_r
\quad\text{for some }2\le r\le n-2.
\end{equation}

\subsection{The maximal-parabolic wall strata}

We next describe the \(G_0\)-orbits in
\[
Y^0_{S,K_r},
\qquad
K_r=S\setminus\{\alpha_r\}.
\]
Toeplitz coordinates convert the Bruhat conditions defining
\(Y^0_{S,K_r}\) into a factorization of two complementary polynomials.
This gives an explicit description of its \(G_0\)-orbits.
We recall the Toeplitz parametrization of the Peterson big cell:
\[
\mathrm{Pet}_n\cap Bw_0B/B
=
\left\{
u(\mathbf a)w_0B/B:
\mathbf a=(a_1,\ldots,a_{n-1})\in\mathbb C^{n-1}
\right\},
\]
where
\[
u(\mathbf a)
=
I+\sum_{j=1}^{n-1}a_jN^j,
\qquad
N=\sum_{i=1}^{n-1}E_{i,i+1}.
\]
Write
\[
u(\mathbf a)^{-1}
=
I+\sum_{j=1}^{n-1}b_jN^j.
\]

We shall use the following elementary consequence of the Bruhat rank
criterion. If \(J\subsetneq S\) and the first block determined by
\(J\) has size \(d\), then the first nonzero entry in the first row
of a matrix in \(B^-w_JB\) occurs in column \(d\).

\begin{proposition}\label{prop:maximal-stratum-orbits}
Let \(1\le r\le n-1\). Then the set of \(G_0\)-orbits in
\(Y^0_{S,K_r}\) is naturally identified with $\binom{\mu_n}{r}\big/\mu_n,$
where $\mu_n=\{\zeta\in\mathbb C^*:\zeta^n=1\},$ and 
\(\binom{\mu_n}{r}\) denotes the set of \(r\)-element subsets of
\(\mu_n\), and \(\mu_n\) acts by multiplication.

Consequently,
\begin{equation}\label{eq:necklace-number}
N(n,r)
:=
\#\bigl(Y^0_{S,K_r}/G_0\bigr)
=
\frac1n
\sum_{d\mid\gcd(n,r)}
\varphi(d)
\binom{n/d}{r/d}.
\end{equation}
Moreover,
\begin{equation}\label{eq:single-orbit-normal}
N(n,r)=1
\quad\Longleftrightarrow\quad
r\in\{1,n-1\}
\quad\Longleftrightarrow\quad
K_r\in\mathcal N.
\end{equation}
\end{proposition}

\begin{proof}
Assume \(n\ge4\), and let \(1\le r\le n-1\). We put \(m=n-r\), and let $x=u(\mathbf a)w_0B/B.$
Suppose first that \(x\in Y^0_{S,K_r}\). Then
\(u(\mathbf a)w_0 \in B^-w_{K_r}B\). The first row of \(u(\mathbf a)w_0\) is
\[
(a_{n-1},a_{n-2},\ldots,a_1,1).
\]
Since the first block of \(w_{K_r}\) has size \(r\), the Bruhat
criterion gives that its first
nonzero entry must occur in column \(r\). As \(m=n-r\), this gives
\[
a_{m+1}=\cdots=a_{n-1}=0,
\qquad
a_m\neq0.
\]
We have,
\[
w_0\bigl(u(\mathbf a)w_0\bigr)^{-1}w_0
=
u(\mathbf a)^{-1}w_0.
\]
Moreover,
\[
u(\mathbf a)w_0 \in B^-w_{K_r}B
\quad\Longrightarrow\quad
w_0(u(\mathbf a)w_0)^{-1}w_0
\in
B^-\bigl(w_0w_{K_r}w_0\bigr)B
=
B^-w_{K_{n-r}}B.
\]
Writing $u(\mathbf a)^{-1}
=
I+\sum_{j=1}^{n-1}b_jN^j,$ the first row of \(u(\mathbf a)^{-1}w_0\) is $(b_{n-1},b_{n-2},\ldots,b_1,1).$
Since the first block of \(w_{K_{n-r}}\) has size \(n-r=m\), the
Bruhat criterion again  gives
\[
b_{r+1}=\cdots=b_{n-1}=0,
\qquad
b_r\neq0.
\]
Conversely, suppose these two sets of conditions hold, and let
\(J\subseteq S\) be the unique subset such that
\(x\in B^-w_JB/B\). The first set of conditions shows that the first
block associated with \(J\) has size \(r\), while the second shows
that its last block has size \(n-r\). Since these two sizes add up to
\(n\), the block decomposition is exactly \(n=r+(n-r)\). Hence
\(J=K_r\). Therefore
\begin{equation}\label{eq:maximal-stratum-toeplitz}
x\in Y^0_{S,K_r}
\quad\Longleftrightarrow\quad
\begin{cases}
a_{m+1}=\cdots=a_{n-1}=0,\quad a_m\neq0,\\
b_{r+1}=\cdots=b_{n-1}=0,\quad b_r\neq0.
\end{cases}
\end{equation}

Consider the polynomials
\[
A(T)=1+a_1T+\cdots+a_mT^m,
\qquad
B(T)=1+b_1T+\cdots+b_rT^r.
\]
Since $u(\mathbf a)u(\mathbf a)^{-1}=I$
and \(N^n=0\), we have 
\[
A(T)B(T)\equiv1\pmod{T^n}.
\]
Since $\deg A=m,\deg B=r$ and $ m+r=n$, 
we obtain
\begin{equation}\label{eq:AB-factorization}
A(T)B(T)=1+\gamma T^n
\end{equation}
for some $\gamma=a_mb_r\neq0.$

Suppose $B(T)$ factors as
\[
B(T)=\prod_{i=1}^r(1-z_iT).
\]
Since \(b_r\neq0\), all \(z_i\) are nonzero. By
\eqref{eq:AB-factorization}, we have $B(T)\mid 1+\gamma T^n.$
The polynomial \(1+\gamma T^n\) is separable, since
\(\gamma\neq0\). Hence \(B(T)\) is separable, so the \(z_i\)'s are
pairwise distinct. Substituting \(T=z_i^{-1}\) in
\eqref{eq:AB-factorization} gives
\[
z_i^n=-\gamma,
\qquad
1\le i\le r.
\]
Thus every point of \(Y^0_{S,K_r}\) determines an unordered set $\{z_1,\ldots,z_r\}
\subset\mathbb C^*$
of distinct elements having the same nonzero \(n\)-th power.

Conversely, suppose \(z_1,\ldots,z_r\) are distinct nonzero numbers
such that $z_1^n=\cdots=z_r^n=c\neq0.$
Then $B(T):=\prod_{i=1}^r(1-z_iT)$
divides \(1-cT^n\). Indeed, the roots of \(B(T)\) are the distinct numbers \(z_i^{-1}\),
and
\[
1-cz_i^{-n}=1-\frac{c}{z_i^n}=0.
\]
Hence every linear factor \(1-z_iT\) divides \(1-cT^n\), and since
these factors are pairwise coprime, their product \(B(T)\) divides
\(1-cT^n\). We put $A(T):=\frac{1-cT^n}{B(T)}.$
Then \(A(0)=1\), $\deg A=n-r=m,$ $\deg B=r$, and both leading coefficients are nonzero. Hence
\[
A(T)B(T)\equiv1\pmod{T^n},
\]
and \eqref{eq:maximal-stratum-toeplitz} gives a point of
\(Y^0_{S,K_r}\). We have therefore obtained a bijection between
\(Y^0_{S,K_r}\) and the unordered \(r\)-element subsets $\{z_1,\ldots,z_r\}\subset\mathbb C^*$
whose elements have a common nonzero \(n\)-th power.

We choose \(s\in\mathbb C^*\)
such that \(s^n=c\). Then
\[
C:=\{s^{-1}z_1,\ldots,s^{-1}z_r\}
\in\binom{\mu_n}{r},
\qquad
\{z_1,\ldots,z_r\}=sC.
\]
A different choice of \(s\) changes \(C\) by multiplication by an
element of \(\mu_n\). Under the \(G_0\)-action, the coefficients of
\[
B(T)=1+b_1T+\cdots+b_rT^r
\]
transform as \(b_j\mapsto t^{-j}b_j\). Thus \(B(T)\) transforms into
\[
1+\sum_{j=1}^r t^{-j}b_jT^j
=
B(t^{-1}T)
=
\prod_{i=1}^r\bigl(1-(t^{-1}z_i)T\bigr).
\]
Hence the associated subset transforms as $\{z_1,\ldots,z_r\}
\longmapsto
\{t^{-1}z_1,\ldots,t^{-1}z_r\}.$
Every \(G_0\)-orbit therefore has a representative in
\(\binom{\mu_n}{r}\), and two such representatives \(C,C'\) lie in
the same orbit if and only if $C'=\xi C$ for some $\xi\in\mu_n.$
Consequently,
\[
Y^0_{S,K_r}/G_0
\cong
\binom{\mu_n}{r}/\mu_n.
\]


The formula \eqref{eq:necklace-number} follows from Burnside's lemma.
Indeed, if \(\xi\in\mu_n\) has order \(d\), then its action on
\(\mu_n\) has \(n/d\) cycles, each of size \(d\). An \(r\)-element
subset is fixed by \(\xi\) if and only if \(d\mid r\), in which case
there are $\binom{n/d}{r/d}$
such subsets. Since there are \(\varphi(d)\) elements of order \(d\),
\eqref{eq:necklace-number} follows.

For \(r=1\), multiplication by \(\mu_n\) is transitive on
\(\mu_n\), and the case \(r=n-1\) follows by taking complements.
Hence \(N(n,r)=1\) for \(r=1,n-1\).

If \(2\le r\le n-2\), choose a primitive \(n\)-th root
\(\zeta\) and consider
\[
C_1=\{1,\zeta,\ldots,\zeta^{r-1}\},
\qquad
C_2=\{1,\zeta,\ldots,\zeta^{r-2},\zeta^r\}.
\]
The first contains \(r-1\) cyclically consecutive pairs, while the
second contains \(r-2\). This number is invariant under multiplication
by \(\mu_n\), so \(C_1\) and \(C_2\) belong to distinct
\(\mu_n\)-orbits. Thus \(N(n,r)\ge2\).

Finally, among the maximal subsets \(K_r\), we have
\[
K_r\in\mathcal N
\quad\Longleftrightarrow\quad
r\in\{1,n-1\},
\]
which proves \eqref{eq:single-orbit-normal}.
\end{proof}

Thus, among the maximal-parabolic Peterson cells, the attracting
stratum \(Y^0_{S,K_r}\) is a single \(G_0\)-orbit precisely when
\(X_{K_r}^0\) lies in the normal locus of \(\mathrm{Pet}_n\).

\subsection{Nonnormality on the middle walls}

The preceding orbit description allows us to treat linearizations
lying on a middle wall.

\begin{proposition}\label{prop:middle-wall-nonnormal}
Assume \(n\ge4\), and let \(\chi\) be a regular dominant
root-lattice character. If $\langle w_{K_r}(\chi),\lambda\rangle=0$
for some \(2\le r\le n-2\), then the quotient $(\mathrm{Pet}_n)^{ss}_{G_0}(\mathcal L(\chi))
\sslash G_0$ is not normal.
\end{proposition}

\begin{proof} 
Let
\[
\pi_\chi:
(\mathrm{Pet}_n)^{ss}_{G_0}(\mathcal L(\chi))
\longrightarrow
(\mathrm{Pet}_n)^{ss}_{G_0}(\mathcal L(\chi))
\sslash G_0
\]
be the quotient morphism, and set $q_K:=\pi_\chi(p_K)$, where $K=K_r$.
Since $\langle w_K(\chi),\lambda\rangle=0,$
we have
\[
\mu^{\mathcal L(\chi)}(p_K,\lambda)
=
\mu^{\mathcal L(\chi)}(p_K,-\lambda)
=
0.
\]
Hence \(p_K\) is semistable and, being fixed by \(G_0\), is
polystable.

Since
\(\mathcal C^\lambda_{\mathrm{deep}}\) is a nonempty rational open
cone, we choose a regular dominant character $\eta\in\mathcal C^\lambda_{\mathrm{deep}}$ which is in the root lattice. For an integer \(M\gg0\), we set $\chi_M:=M\chi+\eta.$
Then \(\chi_M\) is again a regular dominant root-lattice character.
Choose \(M\) sufficiently large so that every nonzero number $\langle w_I(\chi),\lambda\rangle$, for every $I\subseteq S,$
has the same sign as $\langle w_I(\chi_M),\lambda\rangle.$ Since $\langle w_I(\eta),\lambda\rangle<0$ for $I\subsetneq S,$
every pairing which vanishes at \(\chi\) becomes strictly negative
at \(\chi_M\). In particular,
\begin{equation}\label{eq:K-negative-perturbation}
\langle w_K(\chi_M),\lambda\rangle<0.
\end{equation}

We claim that
\begin{equation}\label{eq:ss-M-inclusion}
(\mathrm{Pet}_n)^{ss}_{G_0}(\mathcal L(\chi_M))
\subseteq
(\mathrm{Pet}_n)^{ss}_{G_0}(\mathcal L(\chi)).
\end{equation}
Indeed, let $x\in Y^0_{I,J}$ be \(\chi_M\)-semistable. By
Lemma~\ref{lem:stability-strata},
\[
\langle w_I(\chi_M),\lambda\rangle\ge0,
\qquad
\langle w_J(\chi_M),\lambda\rangle\le0.
\]

If \(I\subsetneq S\), then
\[
\langle w_I(\chi_M),\lambda\rangle
=
M\langle w_I(\chi),\lambda\rangle
+
\langle w_I(\eta),\lambda\rangle.
\]
Since the second term is strictly negative, the inequality $\langle w_I(\chi_M),\lambda\rangle\ge0$ implies $\langle w_I(\chi),\lambda\rangle>0.$
If \(I=S\), then $\langle w_0(\chi),\lambda\rangle>0$
automatically holds. On the other hand, if $\langle w_J(\chi),\lambda\rangle>0,$
then, by the choice of \(M\),
\[
\langle w_J(\chi_M),\lambda\rangle>0,
\]
contradicting semistability. Hence $\langle w_J(\chi),\lambda\rangle\le0.$
Thus \(x\) is \(\chi\)-semistable, proving
\eqref{eq:ss-M-inclusion}.

Let
\[
\pi_M:
(\mathrm{Pet}_n)^{ss}_{G_0}(\mathcal L(\chi_M))
\longrightarrow
(\mathrm{Pet}_n)^{ss}_{G_0}(\mathcal L(\chi_M))
\sslash G_0
\]
be the corresponding good quotient. By
\eqref{eq:ss-M-inclusion}, the restriction $\pi_\chi\big|_
{(\mathrm{Pet}_n)^{ss}_{G_0}(\mathcal L(\chi_M))}$
is \(G_0\)-invariant. Hence, by the categorical property of the good
quotient, it factors uniquely through \(\pi_M\), giving a morphism
\[
f_M:
(\mathrm{Pet}_n)^{ss}_{G_0}(\mathcal L(\chi_M))
\sslash G_0
\longrightarrow
(\mathrm{Pet}_n)^{ss}_{G_0}(\mathcal L(\chi))
\sslash G_0
\]
such that
\[
f_M\circ\pi_M
=
\pi_\chi
\]
on
\((\mathrm{Pet}_n)^{ss}_{G_0}(\mathcal L(\chi_M))\).

Since the quotients are projective varieties, we get that \(f_M\) is
projective, and hence proper. Moreover, \(f_M\) is birational:
the dense open Richardson stratum $Y^0_{S,\emptyset}$
is stable for every regular dominant linearization, and its geometric
quotient is a common dense open subset of the two GIT quotients.

We next determine the fiber of \(f_M\) over \(q_K\). We claim,
set-theoretically, that
\begin{equation}\label{eq:wall-fiber}
f_M^{-1}(q_K)
=
Y^0_{S,K}/G_0.
\end{equation}

Let $y\in f_M^{-1}(q_K),$
and we choose $x\in
(\mathrm{Pet}_n)^{ss}_{G_0}(\mathcal L(\chi_M))$ such that $\pi_M(x)=y.$
Then
\[
\pi_\chi(x)
=
f_M(\pi_M(x))
=
q_K
=
\pi_\chi(p_K).
\]
Thus \(x\) and \(p_K\) are \(S\)-equivalent for the
\(\chi\)-linearization. Since \(\{p_K\}\) is a closed semistable
orbit, \(p_K\) lies in the closure of the \(G_0\)-orbit of \(x\). On the other hand, \(p_K\) is not \(\chi_M\)-semistable, because
\(\langle w_K(\chi_M),\lambda\rangle<0\), whereas \(x\) is
\(\chi_M\)-semistable. Thus \(x\neq p_K\). Moreover, \(x\) cannot be
fixed by \(G_0\), since otherwise
\(\overline{G_0\cdot x}=\{x\}\) could not contain \(p_K\).
Consequently, \(G_0\cdot x\) is one-dimensional.

Suppose $x\in Y^0_{I,J}.$ Then by Proposition~\ref{thm:BB-Peterson}, we have 
\[
\lim_{t\to0}\lambda(t)\cdot x=p_J,
\qquad
\lim_{t\to\infty}\lambda(t)\cdot x=p_I.
\]
The closure of the one-dimensional \(G_0\)-orbit of \(x\) consists
of the orbit together with these two fixed limits. Since \(p_K\)
belongs to this closure, we must have $I=K$ or $J=K$.
The first possibility is impossible, because
\eqref{eq:K-negative-perturbation} and
Lemma~\ref{lem:stability-strata} show that no stratum with first
index \(K\) is \(\chi_M\)-semistable. Hence $J=K.$
Since \(J\subseteq I\) and \(K\) is a maximal proper subset of \(S\),
we have $I=K$ or $I=S$. Again \(I=K\) is impossible, so $x\in Y^0_{S,K}.$

Conversely, let $x\in Y^0_{S,K}.$ Since 
\[
\langle w_0(\chi_M),\lambda\rangle>0
\]
and
\[
\langle w_K(\chi_M),\lambda\rangle<0,
\]
by Lemma~\ref{lem:stability-strata} we get that \(x\) is
\(\chi_M\)-stable. Moreover, $\lim_{t\to0}\lambda(t)\cdot x=p_K.$
Since \(p_K\) is \(\chi\)-semistable, the orbit closure of \(x\) in
the \(\chi\)-semistable locus meets the closed orbit \(\{p_K\}\).
Hence
\[
\pi_\chi(x)=\pi_\chi(p_K)=q_K.
\]
This proves \eqref{eq:wall-fiber}.

Since \(Y^0_{S,K}\) is contained in the \(\chi_M\)-stable locus, its
quotient by \(G_0\) is geometric.
By
Proposition~\ref{prop:maximal-stratum-orbits},
\[ \#(f_M^{-1}(q_K))=
\#\bigl(Y^0_{S,K}/G_0\bigr)
=
N(n,r)\ge2
\]
for \(2\le r\le n-2\). Hence the fiber
\(f_M^{-1}(q_K)\) is a finite set with at least two points, and is
therefore disconnected.

Suppose, for contradiction, that $(\mathrm{Pet}_n)^{ss}_{G_0}(\mathcal L(\chi))
\sslash G_0$
is normal. By the Stein factorization theorem
\cite[Chapter~III, Corollary~11.5]{Hart}, the proper morphism \(f_M\)
factors as
\[
(\mathrm{Pet}_n)^{ss}_{G_0}(\mathcal L(\chi_M))
\sslash G_0
\overset{g}{\longrightarrow}
Z
\overset{h}{\longrightarrow}
(\mathrm{Pet}_n)^{ss}_{G_0}(\mathcal L(\chi))
\sslash G_0,
\]
where \(g\) has connected fibers and \(h\) is finite.

Since \(f_M\) is birational, \(h\) is finite and birational. The
target is normal by assumption, so \(h\) is an isomorphism. Hence
\(f_M=g\) has connected fibers. This contradicts the disconnectedness
of \(f_M^{-1}(q_K)\). Therefore $(\mathrm{Pet}_n)^{ss}_{G_0}(\mathcal L(\chi))
\sslash G_0$
is not normal.
\end{proof}

\subsection{The normality criterion} We now combine the description of the normal locus with the analysis of
positive and zero middle weights to obtain the complete normality
criterion.

\begin{theorem}\label{thm:normality-criterion} 
Assume \(n\ge4\), and let \(\chi\) be a regular dominant character in the root lattice. Then the following are equivalent:
\begin{enumerate}
\item
\[
(\mathrm{Pet}_n)^{ss}_{G_0}(\mathcal L(\chi))
\sslash G_0
\]
is normal;

\item
\[
\langle w_{K_r}(\chi),\lambda\rangle<0
\qquad
(2\le r\le n-2);
\]

\item
\[
(\mathrm{Pet}_n)^{ss}_{G_0}(\mathcal L(\chi))
\subseteq
\mathrm{Pet}_n^{\mathrm{norm}}.
\]
\end{enumerate}
\end{theorem}

\begin{proof}
We first prove \((2)\Longleftrightarrow(3)\).

Assume (2), and let \(I\notin\mathcal N\). By
\eqref{eq:bad-index-normality}, there is some \(2\le r\le n-2\)
such that $I\subseteq K_r.$ Hence \(w_I\le w_{K_r}\). Since \(\chi\) is dominant and
\(-\lambda\) is strictly dominant,
\[
\langle w_I(\chi),\lambda\rangle
\le
\langle w_{K_r}(\chi),\lambda\rangle
<0.
\]
By Lemma~\ref{lem:stability-strata}, no Richardson stratum with first index \(I\) is semistable.
Using \eqref{eq:normal-index-set}, we obtain
\[
(\mathrm{Pet}_n)^{ss}_{G_0}(\mathcal L(\chi))
\subseteq
\mathrm{Pet}_n^{\mathrm{norm}}.
\]

Conversely, suppose (2) fails. Then $\langle w_{K_r}(\chi),\lambda\rangle\ge0$
for some \(2\le r\le n-2\). If this pairing is positive, the
nonempty stratum $Y^0_{K_r,\emptyset}\subseteq X_{K_r}^0$ is semistable, since $\langle\chi,\lambda\rangle<0.$
If the pairing is zero, the fixed point $p_{K_r}=Y^0_{K_r,K_r}$
is semistable. In either case the semistable locus meets
\(X_{K_r}^0\), which lies in the non-normal locus by
Proposition~\ref{prop:normal-locus}. Hence (3) fails. Thus
\[
(2)\Longleftrightarrow(3).
\]

We next prove \((3)\Rightarrow(1)\). Under (3), the semistable locus
is an open subvariety of \(\mathrm{Pet}_n^{\mathrm{norm}}\), and is
therefore normal. Let
\[
\pi_\chi:
(\mathrm{Pet}_n)^{ss}_{G_0}(\mathcal L(\chi))
\longrightarrow
(\mathrm{Pet}_n)^{ss}_{G_0}(\mathcal L(\chi))
\sslash G_0
\]
be the quotient morphism. If \(Q\) is an affine open subset of
the quotient and $\pi_\chi^{-1}(Q)=\operatorname{Spec}A,$
then $Q=\operatorname{Spec}A^{G_0}.$
The ring \(A\) is integrally closed. If $f\in\operatorname{Frac}(A^{G_0})$ is integral over \(A^{G_0}\), then it is integral over \(A\), hence
\(f\in A\); since \(f\) is \(G_0\)-invariant, \(f\in A^{G_0}\).
Thus \(A^{G_0}\) is integrally closed, proving (1).

It remains to prove \((1)\Rightarrow(2)\). Suppose that (2) fails.

If $\langle w_{K_r}(\chi),\lambda\rangle=0$ for some \(2\le r\le n-2\), then
Proposition~\ref{prop:middle-wall-nonnormal} shows directly that the
quotient is non-normal.

It remains to consider the case
\begin{equation}\label{eq:positive-middle-normality}
\langle w_{K_r}(\chi),\lambda\rangle>0
\end{equation}
for some \(2\le r\le n-2\). Put \(K=K_r\). Then $Y^0_{K,\emptyset}
=
X_K^0\cap B^-B/B$
is a nonempty open subset of the irreducible affine cell \(X_K^0\).
Moreover, by \eqref{eq:positive-middle-normality} and
\(\langle\chi,\lambda\rangle<0\), Lemma~\ref{lem:stability-strata}
gives
\[
Y^0_{K,\emptyset}
\subseteq
(\mathrm{Pet}_n)^s_{G_0}(\mathcal L(\chi)).
\]
Since \(K\notin\mathcal N\),
\[
X_K^0
\subseteq
\mathrm{Pet}_n\setminus\mathrm{Pet}_n^{\mathrm{norm}}.
\]

We first produce a point of \(Y^0_{K,\emptyset}\) with trivial
stabilizer. Since \(2\le r\le n-2\), we have \(\alpha_1\in K\), and
the closure relation gives $X_{\{\alpha_1\}}^0\subseteq X_K.$
On
\[
X_{\{\alpha_1\}}^0
=
\{u_{\alpha_1}(z)s_1B/B:z\in\mathbb C\}
\]
the action is $z\longmapsto t^{-1}z.$
Hence the \(G_0\)-action on \(X_K\) is effective. Since \(X_K^0\) is
dense in \(X_K\), the induced action on \(X_K^0\) is also effective.

We write
\[
X_K^0=\operatorname{Spec}A,
\qquad
A=\bigoplus_{m\in\mathbb Z}A_m,
\]
and choose nonzero homogeneous algebra generators
\(f_1,\ldots,f_N\) of weights \(m_1,\ldots,m_N\).
Effectiveness implies
\[
\gcd\{m_i:m_i\neq0\}=1;
\]
otherwise the corresponding nontrivial finite subgroup of
\(\mathbb G_m\) would act trivially on every generator.
Since \(A\) is a domain, $U:=
\bigcap_{m_i\neq0}D(f_i)$ is a nonempty \(G_0\)-invariant open subset. If \(t\in G_0\) fixes
\(x\in U\), then $t^{m_i}=1, \,\, m_i\neq0$
and hence \(t=1\). Thus every point of \(U\) has trivial stabilizer.

Since \(U\) and \(Y^0_{K,\emptyset}\) are nonempty open subsets of
the irreducible variety \(X_K^0\), choose $x\in U\cap Y^0_{K,\emptyset}.$
Then \(x\) is stable, is a non-normal point of \(\mathrm{Pet}_n\),
and satisfies $\operatorname{Stab}_{G_0}(x)=\{1\}.$

Suppose, for contradiction, that $(\mathrm{Pet}_n)^{ss}_{G_0}(\mathcal L(\chi))
\sslash G_0$
is normal. Choose an affine open neighborhood \(Q\) of
\(\pi_\chi(x)\), contained in the image of the stable locus. We put
\[
W:=\pi_\chi^{-1}(Q)=\operatorname{Spec}B.
\]
Then \(W\) is contained in the stable locus and $Q=\operatorname{Spec}B^{G_0};$
in particular, the quotient \(W\to Q\) is geometric.

We write $B=\bigoplus_{d\in\mathbb Z}B_d.$
Since \(\operatorname{Stab}_{G_0}(x)=\{1\}\), the weights \(d\) for
which there exists \(g\in B_d\) with \(g(x)\neq0\) have greatest
common divisor \(1\). Hence we may choose homogeneous functions
\[
g_i\in B_{d_i},
\qquad
g_i(x)\neq0,
\qquad
\gcd(d_1,\ldots,d_s)=1.
\]
Set
\[
V:=D(g_1\cdots g_s)\subseteq W.
\]
We choose integers \(c_i\) such that $\sum_i c_id_i=1.$
Since the \(g_i\)'s are invertible on \(V\), the function $h:=\prod_i g_i^{c_i}$ is an invertible homogeneous element of degree \(1\). Therefore
\[
\mathcal O(V)
=
\mathcal O(V)_0[h,h^{-1}],
\]
since for every homogeneous $a\in\mathcal O(V)_d$, $ah^{-d}\in\mathcal O(V)_0.$
Consequently,
\[
V
\simeq
\operatorname{Spec}\mathcal O(V)_0\times\mathbb G_m.
\]
The open subset \(V\) is \(G_0\)-invariant. Since \(W\to Q\) is a
geometric quotient, its fibers are precisely the \(G_0\)-orbits. The invariant open subset \(V\) is a union of
such fibers and is therefore saturated. Thus \(\pi_\chi(V)\) is open in \(Q\) and
\[
\pi_\chi(V)
\simeq
\operatorname{Spec}\mathcal O(V)_0.
\]
Normality of the quotient implies that \(\pi_\chi(V)\) is normal,
and therefore
\[
V
\simeq
\pi_\chi(V)\times\mathbb G_m
\]
is normal. This is impossible, since \(x\in V\) is a non-normal point
of \(\mathrm{Pet}_n\), and \(V\) is an open neighborhood of \(x\) in
\(\mathrm{Pet}_n\). Thus the quotient is non-normal whenever
\eqref{eq:positive-middle-normality} holds. Together with
Proposition~\ref{prop:middle-wall-nonnormal}, this proves
\((1)\Rightarrow(2)\), and hence the three conditions are equivalent.
\end{proof}

\begin{corollary}\label{cor:normality-small-rank}
For \(n\le5\), every Peterson GIT quotient associated with a regular
dominant root-lattice character is normal. For \(n\ge6\), both normal
and non-normal Peterson GIT quotients occur as the regular dominant
root-lattice character varies.
\end{corollary}

\begin{proof}
For \(n\le3\), 
\(\mathrm{Pet}_n\) is normal. Hence every semistable locus is normal,
and so is its good quotient by \(G_0\).

Let \(n=4\) or \(5\). By
Lemma~\ref{lem:maximal-parabolic-walls}, we have $\langle w_{K_r}(\chi),\lambda\rangle<0$ for $2\le r\le n-2$ and 
for every regular dominant character \(\chi\). Therefore
by Theorem~\ref{thm:normality-criterion} we get that the
quotient is normal.

Now let \(n\ge6\). The deep chamber is nonempty, and every
root-lattice character
\(\chi\in\mathcal C^\lambda_{\mathrm{deep}}\) satisfies
\[
\langle w_{K_r}(\chi),\lambda\rangle<0
\qquad
(2\le r\le n-2).
\]
Hence Theorem~\ref{thm:normality-criterion} gives normal quotients.

On the other hand, by Lemma~\ref{lem:maximal-parabolic-walls}, the hyperplane
\[
H_{K_2}
=
\left\{
\chi\in\mathcal D:
\langle w_{K_2}(\chi),\lambda\rangle=0
\right\}
\]
meets the regular dominant cone \(\mathcal D\). Since \(H_{K_2}\) is
a rational hyperplane, \(H_{K_2}\cap\mathcal D\) contains a rational
point. Multiplying by a sufficiently divisible positive integer gives
a regular dominant root-lattice character on \(H_{K_2}\). By
Proposition~\ref{prop:middle-wall-nonnormal}, the corresponding quotient is non-normal.
\end{proof}

\begin{corollary}\label{cor:generic-wall-normalization}
Let \(2\le r\le n-2\), and let \(\chi\) be a regular dominant
root-lattice character satisfying $\langle w_{K_r}(\chi),\lambda\rangle=0$
and $\langle w_{K_s}(\chi),\lambda\rangle<0$ for $2\le s\le n-2, \,\,  and \,\, s\neq r$.
Assume moreover that
\[
\langle w_I(\chi),\lambda\rangle\neq0
\qquad
\text{for every }I\subseteq S,\ I\neq K_r.
\]
Thus \(\chi\) lies in the relative interior of the single GIT wall
\(H_{K_r}\).

Let \(\chi_-\) be a regular dominant root-lattice character in the
adjacent chamber, chosen sufficiently close to \(\chi\), such that $\langle w_{K_r}(\chi_-),\lambda\rangle<0.$
Then the VGIT morphism
\[
\nu:
(\mathrm{Pet}_n)^{ss}_{G_0}(\mathcal L(\chi_-))
\sslash G_0
\longrightarrow
(\mathrm{Pet}_n)^{ss}_{G_0}(\mathcal L(\chi))
\sslash G_0
\]
is the normalization morphism.

More precisely, let
\[
\pi_\chi:
(\mathrm{Pet}_n)^{ss}_{G_0}(\mathcal L(\chi))
\longrightarrow
(\mathrm{Pet}_n)^{ss}_{G_0}(\mathcal L(\chi))
\sslash G_0
\]
be the good quotient morphism and put $
q_{K_r}:=\pi_\chi(p_{K_r}).$
Then
\[
\#\nu^{-1}(q_{K_r})
=
N(n,r)
=
\frac1n
\sum_{d\mid\gcd(n,r)}
\varphi(d)
\binom{n/d}{r/d}.
\]
\end{corollary}

\begin{proof}
Since \(\chi\) lies on no wall other than \(H_{K_r}\), the only
strictly semistable closed \(G_0\)-orbit for the
\(\chi\)-linearization is $\{p_{K_r}\}.$
Indeed, by the description of the polystable locus, a semistable
non-stable closed orbit is a fixed point \(p_I\) satisfying
\[
\langle w_I(\chi),\lambda\rangle=0.
\]
Choose a rational point in the adjacent chamber sufficiently close to
\(\chi\), and then multiply it by a positive integer so that it lies
in the root lattice. Denote the resulting root-lattice character by
\(\chi_-\). Then all numerical
weights which are nonzero at \(\chi\) retain their signs. In
particular,
\[
\langle w_{K_s}(\chi_-),\lambda\rangle<0
\qquad
(2\le s\le n-2).
\]
The variation of GIT gives a projective birational morphism
\[
\nu:
(\mathrm{Pet}_n)^{ss}_{G_0}(\mathcal L(\chi_-))
\sslash G_0
\longrightarrow
(\mathrm{Pet}_n)^{ss}_{G_0}(\mathcal L(\chi))
\sslash G_0.
\]

Away from the S-equivalence class of \(p_{K_r}\), all semistable
points for the wall linearization are stable and have the same
stability behavior for \(\chi_-\). Hence \(\nu\) restricts to an
isomorphism
\[
\nu^{-1}\!\left(
\left(
(\mathrm{Pet}_n)^{ss}_{G_0}(\mathcal L(\chi))
\sslash G_0
\right)\setminus\{q_{K_r}\}
\right)
\xrightarrow{\;\sim\;}
\left(
(\mathrm{Pet}_n)^{ss}_{G_0}(\mathcal L(\chi))
\sslash G_0
\right)\setminus\{q_{K_r}\}.
\]

The fiber over \(q_{K_r}\) is computed exactly as in
\eqref{eq:wall-fiber} gives $\nu^{-1}(q_{K_r})
=
Y^0_{S,K_r}/G_0.$
Then by Proposition~\ref{prop:maximal-stratum-orbits}, we get that 
\[
\#\bigl(Y^0_{S,K_r}/G_0\bigr)
=
N(n,r)
=
\frac1n
\sum_{d\mid\gcd(n,r)}
\varphi(d)
\binom{n/d}{r/d}.
\]
Thus every fiber of \(\nu\) is finite. Hence \(\nu\) is
quasi-finite. Since \(\nu\) is projective, it is finite.

By the choice of \(\chi_-\), we get $\langle w_{K_s}(\chi_-),\lambda\rangle<0$ for $2\le s\le n-2)$. So byTheorem~\ref{thm:normality-criterion} we get that $(\mathrm{Pet}_n)^{ss}_{G_0}(\mathcal L(\chi_-))
\sslash G_0$
is normal. The morphism \(\nu\) is therefore finite and birational
with normal source. Hence it is the normalization morphism of $(\mathrm{Pet}_n)^{ss}_{G_0}(\mathcal L(\chi))
\sslash G_0.$
\end{proof}

\section{Smoothness of the Quotient}
In contrast with normality, smoothness is independent of the GIT
chamber. We show that every Peterson GIT quotient is smooth in ranks
\(n\le3\), whereas for \(n\ge4\) every regular dominant
linearization produces a singular quotient. 

\begin{theorem}\label{thm:smoothness} Let $\chi$ be a regular dominant character of \(T\) lying in the root lattice. Then the quotient $(\mathrm{Pet}_n)^{ss}_{G_0}(\mathcal L(\chi))\sslash G_0$ is smooth if and only if \(n\le3\).
\end{theorem}

\begin{proof}
First suppose \(n\le3\). For \(n=2\), the regular dominant cone is
contained in the deep chamber, and by Theorem~\ref{thm:deep-quotient} we get that $(\mathrm{Pet}_2)^{ss}_{G_0}(\mathcal L(\chi))\sslash G_0$ is a point. For \(n=3\), the entire regular dominant cone is the
deep chamber, so
\[
(\mathrm{Pet}_3)^{ss}_{G_0}(\mathcal L(\chi))\sslash G_0
\cong
\mathbb P(1,2)\cong\mathbb P^1.
\]
Thus the quotient is smooth for \(n\le3\).

Now assume \(n\ge 4\). We show that the quotient is singular. Put
\[
r=\left\lfloor\frac n2\right\rfloor,
\qquad
m=n-r.
\]
We work in the Peterson big cell
\[
X_S^0=(Bw_0B/B)\cap \mathrm{Pet}_n\cong \mathbb A^{n-1},
\]
with standard coordinates $(a_1,\ldots,a_{n-1}).$
The \(\lambda\)-action is diagonal:
\[
\lambda(t)\cdot a_k=t^{-k}a_k.
\]
Consider the point \(x\in X_S^0\) corresponding to the
Peterson big-cell coordinates
\[
a_m=1,\qquad a_k=0\quad(k\neq m).
\]
Equivalently,
\[
x=uB/B,
\qquad
u=
\left(I+\sum_{i=1}^{r}E_{i,i+m}\right)w_0.
\]
Thus \(x\in X_S^0\subset \mathrm{Pet}_n\).

We claim that \(x\in Y^0_{S,I}\) for some $I\subseteq S\setminus\{\alpha_r\}.$ Since \(x\in X_S^0\), it lies in \(Y^0_{S,I}\) for a unique subset
\(I\subseteq S\). It remains to show that \(\alpha_r\notin I\). The leading \(r\times r\) principal block of \(u\) is the antidiagonal
permutation matrix.
Hence
\[
\Delta_r(u)=\pm 1\neq 0.
\]
By the Bruhat rank criterion for the opposite Schubert cell
\(B^-w_IB/B\), the rank of the leading \(r\times r\) block is
\[
\#\{1\leq a\leq r\mid w_I(a)\leq r\}.
\]
Since this rank is \(r\), we have
\[
w_I(\{1,\ldots,r\})=\{1,\ldots,r\}.
\]
So \(w_I\in W_{K_r}\), where $K_r=S\setminus\{\alpha_r\}.$
Since the support of the longest element \(w_I\) is exactly \(I\),
it follows that $I\subseteq K_r.$ Therefore $x\in Y^0_{S,I}$
for some \(I\subseteq S\setminus\{\alpha_r\}\).

Since \(I\subseteq K_r\), we have \(w_I\le w_{K_r}\) in Bruhat
order. As \(\chi\) is dominant and \(-\lambda\) is strictly dominant, by Lemma~\ref{lem:maximal-parabolic-walls}, we have
\[
\langle w_I(\chi),\lambda\rangle
\le
\langle w_{K_r}(\chi),\lambda\rangle
<0.
\]
Also,
\[
\bigl\langle w_S(\chi),\lambda\bigr\rangle
=
\bigl\langle w_0(\chi),\lambda\bigr\rangle>0.
\]
Therefore, by the stability criterion for the strata \(Y^0_{S,I}\), $
x\in (\mathrm{Pet}_n)^s_{G_0}(\mathcal L(\chi)).$

Moreover,
\[
\lambda(t)\cdot x=x
\quad\Longleftrightarrow\quad
t^{-m}=1.
\]
Thus $
\operatorname{Stab}_{G_0}(x)=\mu_m.$ Let
\[
\pi:
(\mathrm{Pet}_n)^{ss}_{G_0}(\mathcal L(\chi))
\longrightarrow
(\mathrm{Pet}_n)^{ss}_{G_0}(\mathcal L(\chi))\sslash G_0
\]
be the quotient morphism. Since \(x\) is stable, its \(G_0\)-orbit is
closed in the semistable locus. Moreover, $x\in X_S^0\subseteq\mathrm{Pet}_n^{\mathrm{sm}},$
so \(x\) is a smooth point of \(\mathrm{Pet}_n\), and $\operatorname{Stab}_{G_0}(x)=\mu_m.$

We choose an affine open neighbourhood \(Q\) of \(\pi(x)\) in the
quotient and set $W:=\pi^{-1}(Q).$
Since \(\pi\) is a good quotient, \(W\) is a \(G_0\)-invariant affine
open neighbourhood of \(G_0\cdot x\). We may therefore apply Luna's
\'etale slice theorem \cite{Luna} at the closed orbit \(G_0\cdot x\).
Since \(x\) is smooth, the quotient is \'etale-locally at \(\pi(x)\)
modeled on
\[
N_x\sslash\mu_m,
\qquad
N_x=
T_x\mathrm{Pet}_n/T_x(G_0\cdot x).
\]
Consequently, if the GIT quotient were smooth at \(\pi(x)\), then
\(N_x\sslash\mu_m\) would be smooth at the image of the origin.

We now describe the $\mu_m$-representation $N_x$. 
In the Peterson big-cell coordinates $(a_1,\ldots,a_{n-1}),$
the action is
\[
\lambda(t)\cdot(a_1,\ldots,a_{n-1})
=
(t^{-1}a_1,\ldots,t^{-(n-1)}a_{n-1}).
\]
At the point \(x\), only \(a_m\) is nonzero. Hence the tangent space
to the \(G_0\)-orbit is the \(a_m\)-direction, and therefore
\[
N_x
\cong
\bigoplus_{\substack{1\le k\le n-1\\k\neq m}}
\mathbb C a_k.
\]
For \(t\in\mu_m\), the induced action is
\[
t\cdot a_k=t^{-k}a_k.
\]
We claim that no nontrivial element of \(\mu_m\) acts as a
pseudoreflection on \(N_x\). Let \(\zeta\in\mu_m\) be nontrivial.
If \(\zeta\) has order \(2\), then \(\zeta\) acts nontrivially on both
\(a_1\) and \(a_3\). If \(\zeta\) has order greater than \(2\), then
\(\zeta\) acts nontrivially on both \(a_1\) and \(a_2\). Thus every
nonidentity element acts nontrivially in at least two independent
directions, and hence is not a pseudoreflection.

By the Chevalley--Shephard-Todd theorem (see \cite[Chapter~18, \S18-1]{Kane}), $
N_x\sslash \mu_m$
is singular at the image of the origin. Therefore the GIT quotient is singular at \(\pi(x)\). Hence for every \(n\ge 4\), $(\mathrm{Pet}_n)^{ss}_{G_0}(\mathcal L(\chi))\sslash G_0$
is not smooth. So the theorem follows.
\end{proof}

\section{Examples and Further Consequences}

We conclude by illustrating the chamber structure in rank four and by
recording some additional geometric properties of the deep-chamber
quotient.

\subsection{A wall crossing for \(\mathrm{Pet}_4\)}
We exhibit a chamber adjacent to the deep chamber whose quotient
remains normal but is not isomorphic to the deep-chamber weighted
projective space.

\begin{example}\label{examp:not-weighted}
Let $G=\mathrm{GL}(4,\mathbb C)$ and let $\chi=4\epsilon_1+3\epsilon_2+2\epsilon_3-9\epsilon_4$. We show
that the corresponding GIT quotient is normal but is not isomorphic
to \(\mathbb P(1,2,3)\).

We set $I:=\{\alpha_2,\alpha_3\}=K_1.$ As computed in Example~\ref{exam:not-in-deep}, $\langle w_I(\chi),\lambda\rangle=4>0,$
whereas $\langle w_J(\chi),\lambda\rangle<0$
for every proper subset \(J\subsetneq S\) with \(J\neq I\).
In particular, $\chi\notin\mathcal C^\lambda_{\mathrm{deep}}.$

We set
\[
X_\chi
:=
(\mathrm{Pet}_4)^{ss}_{G_0}(\mathcal L(\chi))
\sslash G_0.
\]
Choose a root-lattice character
\(\chi_{\mathrm{deep}}\in\mathcal C^\lambda_{\mathrm{deep}}\), and put
\[
X_{\mathrm{deep}}
:=
(\mathrm{Pet}_4)^{ss}_{G_0}
(\mathcal L(\chi_{\mathrm{deep}}))
\sslash G_0.
\]
By Theorem~\ref{thm:deep-quotient}, we have $X_{\mathrm{deep}}\cong\mathbb P(1,2,3).$

The semistability criterion shows that, compared with the deep
chamber, the only proper subset whose sign changes is \(I\).
Consequently the deep-chamber semistable locus loses the stratum
\(Y^0_{S,I}\), while the \(\chi\)-semistable locus gains
\[
X_I^0\setminus\{p_I\}
=
\bigsqcup_{J\subsetneq I}Y^0_{I,J}.
\]
Thus there is a common \(G_0\)-invariant open subset \(U\) such that
\[
(\mathrm{Pet}_4)^{ss}_{G_0}
(\mathcal L(\chi_{\mathrm{deep}}))
=
U\sqcup Y^0_{S,I},
\]
whereas
\[
(\mathrm{Pet}_4)^{ss}_{G_0}(\mathcal L(\chi))
=
U\sqcup\bigl(X_I^0\setminus\{p_I\}\bigr).
\]

Since \(I=K_1\), Proposition~\ref{prop:maximal-stratum-orbits} gives
\[
Y^0_{S,I}/G_0
\cong
\binom{\mu_4}{1}/\mu_4,
\]
which consists of a single point. Hence \(Y^0_{S,I}\) is a single
\(G_0\)-orbit. Let $q_0\in X_{\mathrm{deep}}$
denote its image. Then
\begin{equation}\label{eq:example-common-open-deep}
X_{\mathrm{deep}}\setminus\{q_0\}
\cong
U/G_0.
\end{equation}

We next consider the new semistable piece. Solving the Peterson
equations on the Schubert cell \(Bw_IB/B\) gives
\[
X_I^0
=
\left\{
g(b_1,b_2)B/B:
(b_1,b_2)\in\mathbb C^2
\right\},
\]
where
\[
g(b_1,b_2)
=
\begin{pmatrix}
1&0&0&0\\
0&b_2&b_1&1\\
0&b_1&1&0\\
0&1&0&0
\end{pmatrix}.
\]
Thus $X_I^0\simeq\mathbb A^2$ and $p_I=g(0,0)B/B.$
The \(G_0\)-action is
\[
\lambda(t)\cdot(b_1,b_2)
=
(t^{-1}b_1,t^{-2}b_2).
\]
Therefore
\[
C
:=
\bigl(X_I^0\setminus\{p_I\}\bigr)/G_0
\cong
\mathbb P(1,2)
\cong
\mathbb P^1.
\]

We claim that \(X_I^0\setminus\{p_I\}\) is closed in the
\(\chi\)-semistable locus. Indeed,
\[
\overline{X_I^0}
=
\bigsqcup_{J\subseteq I}X_J^0.
\]
The point \(p_I\) is unstable since $\langle w_I(\chi),\lambda\rangle>0,$ and, for every \(J\subsetneq I\), $
\langle w_J(\chi),\lambda\rangle<0.$
Hence no point in a boundary cell \(X_J^0\), \(J\subsetneq I\), is
semistable with first index \(J\). Thus
\[
\overline{X_I^0}\cap
(\mathrm{Pet}_4)^{ss}_{G_0}(\mathcal L(\chi))
=
X_I^0\setminus\{p_I\}.
\]
Since the quotient morphism is a good quotient, \(C\) is therefore a
closed irreducible curve in \(X_\chi\). Moreover,
\begin{equation}\label{eq:example-common-open}
X_\chi\setminus C
\cong
U/G_0
\cong
X_{\mathrm{deep}}\setminus\{q_0\}.
\end{equation}

We first note that \(X_\chi\) is normal. For \(n=4\), the only middle
index is \(r=2\), and a direct computation gives
\[
\langle w_{K_2}(\chi),\lambda\rangle=-8<0.
\]
Hence Theorem~\ref{thm:normality-criterion} implies that \(X_\chi\)
is normal.

We now show that $X_\chi\not\cong\mathbb P(1,2,3).$
Both \(X_\chi\) and \(X_{\mathrm{deep}}\) are normal surfaces.
Since \(C\simeq\mathbb P^1\) is an irreducible closed curve,
\(C\) is a prime divisor on \(X_\chi\), whereas \(q_0\) has
codimension \(2\) in \(X_{\mathrm{deep}}\). Hence
\[
\operatorname{Cl}(X_{\mathrm{deep}})
\cong
\operatorname{Cl}
\bigl(X_{\mathrm{deep}}\setminus\{q_0\}\bigr).
\]
Since $X_{\mathrm{deep}}\cong\mathbb P(1,2,3),$ we have
\[
\operatorname{Cl}
\bigl(X_{\mathrm{deep}}\setminus\{q_0\}\bigr)
\cong
\operatorname{Cl}(\mathbb P(1,2,3))
\cong\mathbb Z.
\]
By \eqref{eq:example-common-open},
\[
\operatorname{Cl}(X_\chi\setminus C)\cong\mathbb Z.
\]

The localization sequence for divisor class groups gives
\[
\mathbb Z[C]
\longrightarrow
\operatorname{Cl}(X_\chi)
\longrightarrow
\operatorname{Cl}(X_\chi\setminus C)
\longrightarrow0.
\]
We claim that the class \([C]\) has infinite order. Suppose that $m[C]=0$
for some \(m>0\). Indeed, if \(mC=\operatorname{div}(f)\)
for some \(m>0\), then \(f\) has no zeros or poles on \(X_\chi\setminus C\).
Via the isomorphism
\[
X_\chi\setminus C\cong X_{\mathrm{deep}}\setminus\{q_0\},
\]
the function \(f\) gives an invertible regular function on
\(X_{\mathrm{deep}}\setminus\{q_0\}\). Since \(X_{\mathrm{deep}}\) is normal and \(q_0\) has codimension \(2\),
the invertible regular function extends uniquely to an invertible
regular function on \(X_{\mathrm{deep}}\). Since
\(X_{\mathrm{deep}}\) is projective and irreducible, every global
invertible regular function is constant. Hence \(f\) is constant on
the dense open subset \(X_\chi\setminus C\), and therefore constant
as a rational function on \(X_\chi\), contradicting
\[
\operatorname{div}(f)=mC\neq0.
\]
Therefore \([C]\) has infinite order.

Consequently,
\[
0
\longrightarrow
\mathbb Z[C]
\longrightarrow
\operatorname{Cl}(X_\chi)
\longrightarrow
\mathbb Z
\longrightarrow0
\]
is exact. Since \(\mathbb Z\) is free, this sequence splits, and
hence $\operatorname{Cl}(X_\chi)
\cong
\mathbb Z^2.$
On the other hand,
\[
\operatorname{Cl}(\mathbb P(1,2,3))
\cong\mathbb Z.
\]
Therefore $X_\chi\not\cong\mathbb P(1,2,3).$
\end{example}

\subsection{Properties of the deep-chamber quotient}
Since the deep-chamber quotient is explicitly a weighted projective
space, its normality, smoothness, Cohen-Macaulayness, and
Gorensteinness can be determined directly from the weights.

\begin{remark}
For $\chi\in\mathcal C_{\mathrm{deep}}^\lambda$
a regular dominant character in the root lattice,
Theorem~\ref{thm:deep-quotient} gives
\[
(\mathrm{Pet}_n)^{ss}_{G_0}(\mathcal L(\chi))
\sslash G_0
\cong
\mathcal W_n
:=
\mathbb P(1,2,\ldots,n-1).
\]
\begin{enumerate}
\item[(i)] \textbf{Normality and Cohen--Macaulayness:}  
$\mathcal W_n$ is a toric variety, hence it is normal and Cohen--Macaulay for every $n \ge 2$.

\item[(ii)] \textbf{Smoothness:}  $\mathcal W_n$ is smooth if and only if $n\le3$. For $n\ge4$ the point
$p_2=[0:1:0:\cdots:0]$ has stabiliser $\mu_2$, acting on the slice
$\bigoplus_{k\neq2}\mathbb C a_k$ by $\zeta\cdot a_k=\zeta^{-k}a_k$; it acts by $-1$ on
every odd coordinate, and since $n-1\ge3$ there are at least two of them ($a_1,a_3$).
Hence $\mu_2$ is not generated by pseudo-reflections and the local quotient is singular.

\item[(iii)] \textbf{Gorensteinness:}
$\mathcal W_n$ is Gorenstein if and only if $n\le 4$. For $n=2,3$ this is clear, since
$\mathcal W_{2}$ is a point and $\mathcal W_{3}=\mathbb P(1,2)\cong\mathbb P^{1}$.

Let $n\ge4$. Then $\mathcal W_{n}$ is well-formed: for each weight $i\in\{1,\dots,n-1\}$ the
remaining weights have greatest common divisor $1$, because the weight $1$ survives when
$i\neq1$, and the coprime weights $2$ and $3$ survive when $i=1$. For a well-formed weighted
projective space, Gorensteinness is equivalent to the condition that every weight divides
$S=\sum_{k=1}^{n-1}k=\frac{n(n-1)}{2}$ (see \cite{Watanabe}). For $n=4$ the weights $1,2,3$ all
divide $S=6$, so $\mathcal W_{4}$ is Gorenstein. For $n\ge5$:
\begin{itemize}
\item if $n$ is odd, then $n-1$ is a weight and $S/(n-1)=n/2\notin\mathbb Z$;
\item if $n$ is even, then $n-2$ is a weight and $n-2\mid S$ would give $n-2\mid 2S=n(n-1)$;
since $n\equiv2$ and $n-1\equiv1 \pmod{n-2}$, this forces $n-2\mid 2$, i.e. $n\le4$.
\end{itemize}
Hence $\mathcal W_{n}$ is not Gorenstein for $n\ge5$.
\end{enumerate}

Therefore, the deep chamber quotient, which is isomorphic to $\mathcal W_n$ is always normal and Cohen-Macaulay; it is smooth only for $n\le3$; and it is Gorenstein only for $n\le4$.
\end{remark}

\begin{remark}
For \(n\ge4\), every Peterson GIT quotient considered above is
singular, independently of the GIT chamber. In the deep chamber,
\[
(\mathrm{Pet}_n)^{ss}_{G_0}(\mathcal L(\chi))
\sslash G_0 \cong\mathbb P(1,2,\ldots,n-1)
\]
is normal and singular. Example~\ref{examp:not-weighted} shows that
crossing a wall may preserve normality while increasing the rank of
the divisor class group and destroying the weighted-projective-space
structure.
\end{remark}

\end{document}